\documentclass[11pt,letterpaper]{amsart}

\usepackage[dvipsnames]{xcolor}
\usepackage[a4paper,top=3cm,bottom=2.5cm,left=3cm,right=3cm,marginparwidth=1.75cm]{geometry}

\usepackage[all,cmtip]{xy}
\usepackage{tikz-cd}
\usepackage{mleftright}
\usepackage[all]{xy} 

\usepackage{pgfplots}
\pgfplotsset{compat=1.15}
\usepackage{mathrsfs}
\usetikzlibrary{arrows}

\usepackage{comment}

\usepackage{mathtools}
\usepackage{amsmath, amssymb, amsfonts, latexsym, mdwlist, amsthm}
\usepackage{subfig}
\usepackage{graphicx}
\usepackage{wrapfig}

\usepackage{float}

\usepackage[bookmarks, colorlinks, breaklinks, pdftitle={},
pdfauthor={}]{hyperref}
\hypersetup{linkcolor=blue,citecolor=blue,filecolor=black,urlcolor=blue}

\usepackage{tikz}
\usetikzlibrary{calc,trees,positioning,arrows,chains,shapes.geometric,%
    decorations.pathreplacing,decorations.pathmorphing,shapes,%
    matrix,shapes.symbols}

\tikzset{
>=stealth',
  punktchain/.style={
    rectangle,
    rounded corners,
    draw=black, thick,
    minimum height=3em,
    text centered,
    on chain},
  line/.style={draw, thick, <-},
  element/.style={
    tape,
    top color=white,
    bottom color=blue!50!black!60!,
    minimum width=8em,
    draw=blue!40!black!90, very thick,
    text width=10em,
    minimum height=3.5em,
    text centered,
    on chain},
  every join/.style={->, thick,shorten >=1pt},
  decoration={brace},
  tuborg/.style={decorate},
  tubnode/.style={midway, right=2pt},
}

\usepackage{paralist}
\setdefaultenum{(a)}{(i)}{}{}
\usepackage{enumitem} 
\usepackage{graphicx}

\usetikzlibrary{patterns}

\renewcommand\;{\hspace{.6pt}}

\newcommand\PP{\mathbb P}
\newcommand\C{\mathbb C}
\newcommand\Q{\mathbb Q}
\newcommand\R{\mathbb R}

\newcommand\Z{\mathbb Z}

\newcommand\cC{\mathcal C}

\newcommand\cH{\mathcal H}
\newcommand\cO{\mathcal O}

\newcommand\cE{\mathcal E}

\newcommand\ch{\operatorname{ch}}
\newcommand\CH{\operatorname{CH}}
\newcommand\td{\operatorname{td}}

\newcommand\Hom{\operatorname{Hom}}
\newcommand\Pic{\operatorname{Pic}}

\newcommand\rk{\operatorname{rk}}
\newcommand\id{\operatorname{id}}
\newcommand\Coh{\operatorname{Coh}}

\newcommand\cok{\operatorname{coker}}

\newcommand\gon{\operatorname{gon}}

\renewcommand\hom{\mathcal H\hspace{-1pt}om}

\newcommand\bn{\operatorname{BN}}

\renewcommand\={\ =\ }

\newcommand\beq[1]{\begin{equation}\label{#1}}
\newcommand\eeq{\end{equation}}
\newcommand\beqa{\begin{eqnarray*}}
\newcommand\eeqa{\end{eqnarray*}}

\makeatletter
\newtheorem*{rep@theorem}{\rep@title}
\newcommand{\newreptheorem}[2]{%
\newenvironment{rep#1}[1]{%
 \def\rep@title{#2 \ref{##1}}%
 \begin{rep@theorem}}%
 {\end{rep@theorem}}}
\makeatother

\newtheorem{Thm}{Theorem}[section]
\newreptheorem{Thm}{Theorem}
\newtheorem{Thm*}{Theorem}
\newtheorem{Prop}[Thm]{Proposition}

\newtheorem{Lem}[Thm]{Lemma}

\newtheorem{Cor}[Thm]{Corollary}
\newreptheorem{Cor}{Corollary}
\newtheorem{Con}[Thm]{Conjecture}
\newtheorem{Conj}{Conjecture}

\newtheorem{thm-int}{Theorem}

\theoremstyle{definition}
\newtheorem{Def-s}[Thm]{Definition}
\newtheorem{Def}[Thm]{Definition}
\newtheorem{Rem}[Thm]{Remark}

\newtheorem{Ex}[Thm]{Example}

\newcommand{\ignore}[1]{}

\begin{document}

\title[Stability conditions on CY3 via BN theory of curves]{Stability conditions on Calabi--Yau threefolds via Brill--Noether theory of curves}


\author{Soheyla Feyzbakhsh}
\address{Department of Mathematics, Imperial College, London SW7 2AZ, United Kingdom}
\email{s.feyzbakhsh@imperial.ac.uk}

\author{Naoki Koseki}
\address{The University of Liverpool, Mathematical Sciences Building, Liverpool L69 7ZL, United Kingdom}
\email{koseki@liverpool.ac.uk}

\author{Zhiyu Liu}
\address{School of Mathematical Sciences, Zhejiang University, Hangzhou 310058, P. R. China}
\email{jasonlzy0617@gmail.com}

\author{Nick Rekuski}
\address{The University of Liverpool, Mathematical Sciences Building, Liverpool L69 7ZL, United Kingdom}
\email{rekuski@liverpool.ac.uk}

\begin{abstract}
Fix a polarised Calabi--Yau threefold $(X,H)$. We reduce a version of the Bayer--Macr\`i--Toda conjecture for $(X,H)$, which ensures the existence of Bridgeland stability conditions on $X$, to verifying a Brill--Noether-type inequality for curves on $X$. We then prove this inequality for a broad class of Calabi--Yau threefolds, including complete intersection Calabi--Yau threefolds in weighted projective spaces.
\end{abstract}

\maketitle

{
\hypersetup{linkcolor=blue}
\setcounter{tocdepth}{1}
}

\section{Introduction}

The notion of Bridgeland stability conditions on triangulated categories was introduced by \cite{bridgeland:stability-condition-on-triangulated-category}. Since then, constructing stability conditions on the derived categories of smooth projective threefolds has become an important problem. By extending the notion of tilt-stability on surfaces, Bayer--Macr\`i--Toda \cite{bayer:bridgeland-stability-conditions-on-threefolds} conjectured a Bogomolov--Gieseker-type inequality involving $\ch_3$, which, as shown in \cite{bayer:the-space-of-stability-conditions-on-abelian-threefolds}, yields an open subset in the space of stability conditions on threefolds. Later, a generalized conjecture was proposed in \cite{macri:stability-fano-3fold,bayer-macri:icm-report} by adding a real $1$-cycle $\Gamma$ to the formulation, which we refer to as \ref{conj:bmt}. Such $\ch_3$-inequalities not only suffice to construct Bridgeland stability conditions on threefolds, but also play crucial roles in enumerative geometry of Calabi--Yau threefolds \cite{toda:bogomolov-counting,feyz:curve-counting,feyz:rank-r-dt-theory-from-0,feyz:rank-r-dt-theory-from-1,feyz:physics-abelian-dt,liu:cast-bound-quintic,liu:cast-bound-3fold} and in birational geometry \cite{bbmt14}.

In recent years, Bayer--Macr\`i--Toda's original conjecture and its variant \ref{conj:bmt} have been established for many threefolds, and we refer readers to \cite[Section 4.1]{bayer-macri:icm-report} for a list. For Calabi--Yau threefolds, however, the only known cases are quintic threefolds \cite{chunyi:stability-condition-quintic-threefold}, $(2,4)$-complete intersection threefolds \cite{liu:bg-ineqaulity-quadratic}, weighted hypersurfaces of degree $6$ or $8$ \cite{koseki:double-triple-solids}, and some \'etale quotients of abelian threefolds \cite{bayer:the-space-of-stability-conditions-on-abelian-threefolds}. Each example requires different calculations, and the lack of a unified approach to \ref{conj:bmt} has hindered further progress.

In the first part of the paper, we reduce \ref{conj:bmt} to verifying a simple linear inequality for slope-stable sheaves. Throughout this section, we fix a polarised Calabi--Yau threefold~$(X,H)$, i.e.~a smooth projective threefold $X$ with an ample divisor $H$ such that $K_X= 0$ and $H^1(X, \cO_X)=0$. For any $\epsilon>0$, we consider the following statement:

\vspace{1mm}

\begin{minipage}{14.2cm}\vspace{2mm}
\begin{enumerate}[label=$\mathbf{BG_3(\epsilon)}$]
\item\label{BG3} \emph{~For any $H$-stable sheaf $E$ on $X$ with $0< \ch_1(E).H^2 \leq \epsilon \ch_0(E)H^3$, 
we have 
\begin{equation*}
    \ch_2(E).H<-\frac{1}{2}\ch_1(E).H^2.
\end{equation*}
}
\end{enumerate}
\end{minipage}

\vspace{1mm}

\begin{Thm}[{Theorem \ref{thm:ch2}}]\label{Thm-BG-BMT}
Assume that there exists $\epsilon>0$ such that \ref{BG3} holds. Then there exists a $1$-cycle $\Gamma$ with $\Gamma.H\geq 0$ such that \emph{\ref{conj:bmt}} holds for $(X, H)$. In particular, there exists a family of geometric\footnote{Here, a stability condition $\sigma$ on $X$ is called \textit{geometric} if the structure sheaf $\cO_x$ of any point $x \in X$ is $\sigma$-stable of the same phase.} Bridgeland stability conditions on the bounded derived category of coherent sheaves on $X$.
\end{Thm}

In practice, we prove \ref{BG3} through a dimensional reduction method. To this end, we introduce a new invariant associated to integral curves, called the \emph{Brill--Noether number}, which may be of independent interest.
\begin{Def}\label{def:bnc}
    Let $C$ be an integral projective curve of arithmetic genus $g$. We define the \emph{Brill--Noether number} $\bn_C$ of $C$ as\footnote{If $C$ is Gorenstein, the interval in the definition can be replaced by $(g-1-t,g-1]$ (cf.~Remark \ref{rmk:left-limit}).}
\begin{equation*}
    \bn_C \coloneqq \lim_{t\to 0^+} \sup \left\{ \frac{h^0(E)}{\rk(E)} \colon \ \text{$E$ is a stable sheaf on $C$ with } \frac{\deg(E)}{\rk(E)}\in (g-1-t,g-1+t)    \right\}. 
\end{equation*}
\end{Def}

Our next theorem shows that establishing \ref{BG3} reduces to studying the invariant $\bn_C$ of a single curve $C$.

\begin{Thm}[{Theorem \ref{thm:main-criterion}(a)}]\label{thm-intro-bn}
Suppose that there exists a smooth surface $S\in |H|$ and a smooth curve $C\in |H|_S|$ such that
\begin{equation*}
        \bn_C< \chi(\cO_X(H)).
\end{equation*}
Then \ref{BG3} holds for some $\epsilon>0$, and in particular, \emph{\ref{conj:bmt}} holds for $(X, H)$ and some $1$-cycle $\Gamma$ with $\Gamma.H\geq 0$.
\end{Thm}

When $H$ is basepoint-free, Theorem \ref{thm-intro-bn} together with Bertini's theorem immediately implies a purely topological criterion for \ref{conj:bmt} (cf.~Corollary \ref{cor:trivial-cor-basepoint-free}).
This criterion applies to the examples in \cite{chunyi:stability-condition-quintic-threefold,koseki:double-triple-solids,liu:bg-ineqaulity-quadratic}. For a version involving singular $S$ and $C$, we refer to Theorem \ref{thm:main-criterion}(b).

To obtain an upper bound for $\bn_C$ as required in Theorem \ref{thm-intro-bn}, we either use the explicit geometry of the curve $C$ or embed it into a suitable surface $S'$, where its Brill--Noether theory can be controlled. As a consequence, we can treat \ref{conj:bmt} uniformly for the following collection of Calabi--Yau threefolds.

\begin{Thm}\label{thm:intro-example}
Conjecture \emph{\ref{conj:bmt}} holds for $(X,H)$ and some $1$-cycle $\Gamma$ with $\Gamma.H\geq 0$ if it belongs to one of the following classes:

\begin{itemize}

 \item \emph{(Theorem \ref{thm:hypergeo}).} $X$ is a general quasi-smooth complete intersection Calabi--Yau threefold in a weighted projective space and $H$ is the generator of $\Pic(X)$.

    \item \emph{(Theorem \ref{thm:cicy-in-fano}).} 
    $X$ is a general divisor in $|-K_M|$ of a Fano fourfold $M$ of index $r$ and $H = \left(-\tfrac{1}{r}K_M\right)\big|_X$, where either $r \geq 3$ or both $r = 2$ and $M$ has Picard number one.

    \item \emph{(Theorem \ref{thm:double-cover}).} $X$ is a cyclic cover $\pi \colon X \to Y$ of a Fano threefold $Y$ of index $r$ and $H = \pi^*\!\left(-\tfrac{1}{r}K_Y\right)$, where either $r\geq 2$ or both $r=1$ and $Y$ has Picard number one.
\end{itemize}
\end{Thm}

We refer to Theorem \ref{thm:double-cover} and Example \ref{ex:high-pic-rk} for more examples of higher Picard numbers. 
We note that Theorem~\ref{thm-intro-bn} is not the only approach to proving \ref{BG3}. For instance, since the statement \ref{BG3} is preserved under taking \'etale covers and \'etale quotients, one may reduce the proof of \ref{BG3} to another simpler Calabi--Yau threefold (cf.~Remark \ref{rmk-etale} and Example \ref{ex:quotient}).

We conjecture that the inequality in Theorem \ref{thm-intro-bn} holds for every Calabi--Yau threefold whose Picard group is generated by a very ample line bundle, see Conjecture \ref{conj:bnC}.

\subsection*{Method of the proof}

In Theorem~\ref{Thm-BG-BMT}, we exploit the freedom to choose a $1$-cycle $\Gamma$ with $\Gamma. H$ sufficiently large. This flexibility allows us to show that a slight strengthening of the classical Bogomolov--Gieseker inequality near slope zero \ref{BG3} already suffices to prove \ref{conj:bmt}, in contrast to the stronger versions used in \cite{chunyi:stability-condition-quintic-threefold,koseki:double-triple-solids,liu:bg-ineqaulity-quadratic} for arbitrary stable sheaves.

The next step is to employ the restriction techniques of \cite{feyz:effective-restriction-theorem} to reduce the proof of \ref{BG3} to the corresponding inequality \hyperref[BGn]{\ensuremath{\mathbf{BG_2(\epsilon)}}} for some divisor $S \in |H|$; see Proposition~\ref{prop:surface-to-3fold-1}. The key difference from \cite{chunyi:stability-condition-quintic-threefold,koseki:double-triple-solids,liu:bg-ineqaulity-quadratic} is that there the restriction of a stable sheaf on $X$ is taken to a divisor in $|2H|$ to guarantee stability. In our case, since we only consider stable sheaves $E$ with slope in the small interval $(0,\epsilon]$, the slopes of the Harder--Narasimhan factors of $E|_S$ can be controlled in terms of $\epsilon$. This allows us to reduce the verification of \ref{BG3} to proving the inequality \hyperref[BGn]{\ensuremath{\mathbf{BG_2(\delta)}}} for some $S \in |H|$ and $\delta>0$.

The final step is to reduce the dimension once more by passing to a curve $C\in |H|_S|$. In this reduction, $\ch_2$ is replaced by the number of sections, and sheaves of slope near zero correspond to stable sheaves on $C$ with slope near $g(C)-1$. This leads to the definition of $\bn_C$ and, ultimately, to the proof of Theorem \ref{thm-intro-bn}.

A logic flow for the proof of Theorem \ref{thm-intro-bn} can be roughly drawn as follows:
$$\bn_C<\chi(\cO_X(H))
~\overset{\mathclap{\text{Prop }\ref{prop:curve-to-surface-smooth}}}{\xRightarrow{\hspace{0.6cm}}}~
\text{\hyperref[BGn]{\ensuremath{\mathbf{BG_2(\delta)}}}}
~\overset{\mathclap{\text{Prop }\ref{prop:surface-to-3fold-1}}}{\xRightarrow{\hspace{0.6cm}}}~
\text{\ref{BG3}}
~\overset{\mathclap{\text{Thm }\ref{thm:ch2}}}{\xRightarrow{\hspace{0.6cm}}}~
Q^{\Gamma}_{0,0}\geq 0
~\xRightarrow{\hspace{0.6cm}}~
\text{\ref{conj:bmt}}$$

\subsection*{Related works}
The original version of Bayer--Macr\`i--Toda conjecture was proposed in \cite[Conjecture 1.3.1]{bayer:bridgeland-stability-conditions-on-threefolds}, and its equivalent form was formulated in \cite[Conjecture 4.1]{bayer:the-space-of-stability-conditions-on-abelian-threefolds}. Currently, \cite[Conjecture 4.1]{bayer:the-space-of-stability-conditions-on-abelian-threefolds} has been proved for many examples, including \cite{bayer:the-space-of-stability-conditions-on-abelian-threefolds,chunyi:fano-3fold,koseki:nef-tangent,liu:bg-ineqaulity-quadratic,macri:projective-space,schmidt:quadric,koseki:product,maciocia:fm-transform-ii}. However, it was discovered in \cite{schmidt:counterexample,koseki:product} that \cite[Conjecture 1.3.1]{bayer:bridgeland-stability-conditions-on-threefolds} fails for certain threefolds. This issue was resolved in \cite[Question 2.4]{macri:stability-fano-3fold} by incorporating an additional $1$-cycle $\Gamma$ into the formulations of \cite[Conjecture 1.3.1]{bayer:bridgeland-stability-conditions-on-threefolds} and \cite[Conjecture 4.1]{bayer:the-space-of-stability-conditions-on-abelian-threefolds} (cf.~Conjecture \ref{conj:bmt}). This generalized version was proved for all Fano threefolds \cite{macri:stability-fano-3fold} and some weighted hypersurfaces \cite{koseki:double-triple-solids}. The only known Calabi--Yau examples for \ref{conj:bmt} are \cite{bayer:the-space-of-stability-conditions-on-abelian-threefolds,chunyi:stability-condition-quintic-threefold,koseki:double-triple-solids,liu:bg-ineqaulity-quadratic}. In this paper, we concentrate on the existence of $\Gamma$ such that \ref{conj:bmt} holds, which suffices for the existence of Bridgeland stability conditions and some applications, e.g.~\cite{feyz:curve-counting,feyz:rank-r-dt-theory-from-0,feyz:rank-r-dt-theory-from-1}, rather than minimizing $\Gamma.H$, which is important for other applications such as \cite{bbmt14,feyz:physics-abelian-dt,liu:cast-bound-quintic}.

After the appearance of the paper, the existence of stability conditions on any projective scheme was verified in \cite{li:remark,lllmpsz}. Later, it was proved in \cite{cheng-fey:threefold} that, at the level of threefolds, this construction agrees with the double-tilt construction of Bayer--Macr\`i--Toda \cite{bayer:bridgeland-stability-conditions-on-threefolds} in a region close to the large-volume limit. As a result, this establishes a much weaker version of the \ref{conj:bmt} inequality. 

Connections between stronger Bogomolov--Gieseker (BG) inequalities for $\ch_2$ of stable sheaves and variants of Conjecture \ref{conj:bmt} have been studied by \cite{chunyi:stability-condition-quintic-threefold,macri:stability-fano-3fold,chunyi:fano-3fold,koseki:double-triple-solids,koseki:hypersurface,liu:bg-ineqaulity-quadratic}. In \cite{chunyi:stability-condition-quintic-threefold}, the Bayer--Macr\`i--Toda conjecture for a quintic threefold is reduced to verifying a stronger BG inequality for all stable sheaves on a $(2,5)$-complete intersection surface. This approach has also been applied to some other Calabi--Yau threefolds by \cite{koseki:double-triple-solids,liu:bg-ineqaulity-quadratic}. On the other hand, the inequality \ref{BG3} involves only stable sheaves of sufficiently small slopes, and hence should be regarded as a slight improvement of the classical BG inequality. Related conjectures for Calabi--Yau threefolds and their hyperplane sections have also been investigated in physics, such as \cite{yau:dry-conj}. It would be very interesting to find a physical derivation of \ref{BG3}.

Bounding numbers of global sections of stable bundles over curves is one of the main topics in the higher rank Brill--Noether theory; we refer to \cite{bigas:brill-noether-for-stable-vector-bundle} for a survey. The picture has been completed for sheaves with small slopes in \cite{newstead:geography-of-brill-noethr-loci,mercat:slope-smaller-2}, and for curves of low genus in \cite{lange:genus-4,lange:genus-5,lange:genus-6}. When a curve can be embedded into the plane, a del Pezzo surface, or a K3 surface, some bounds have been obtained by \cite{feyz-li:clifford-indices,chunyi:stability-condition-quintic-threefold,koseki:double-triple-solids,liu:bg-ineqaulity-quadratic} using wall-crossing of tilt-stability on surfaces. In turn, these results have been used to prove stronger BG inequalities in \cite{chunyi:stability-condition-quintic-threefold,koseki:double-triple-solids,liu:bg-ineqaulity-quadratic}. Thanks to the flexibility of \ref{BG3}, rather than establishing a Brill--Noether-type bound for stable bundles on a curve $C$ with slopes in a wide range as in previous works, we need only focus on an arbitrarily small neighborhood of slope $g(C)-1$.
These considerations lead to our definition of $\bn_C$. For example, even the most naive bound on $\bn_C$ already implies \ref{conj:bmt} for a series of Calabi--Yau threefolds (cf.~Corollary \ref{cor:trivial-cor-basepoint-free} and \ref{thm:hypergeo}), including the examples in \cite{chunyi:stability-condition-quintic-threefold,koseki:double-triple-solids,liu:bg-ineqaulity-quadratic}.

\subsection*{Organization}
In Section \ref{sec:pre}, we introduce a notion of Chern characters of coherent sheaves on singular schemes. This provides a necessary tool for running the machinery of tilt-stability on mildly singular varieties, as we review in Section \ref{subsec:tilt}. Then, in Section \ref{sec:bmt}, we study Conjecture \ref{conj:bmt} for a polarised Calabi--Yau threefold $(X, H)$ and prove Theorem \ref{Thm-BG-BMT} and \ref{thm-intro-bn} via the method described above. Next, in Section \ref{sec:bn-bound}, we first establish various upper bounds on the dimension of global sections of semistable sheaves on integral curves in Sections \ref{subsec:classical-bound} and \ref{subsec:bn-wall-cross}, and then, in Section \ref{subsec:bn-lower-bound}, we discuss lower bounds on $\bn_C$ and its relation to classical invariants. Based on the above results, the verification of Conjecture \ref{conj:bmt} for a series of Calabi--Yau threefolds as in Theorem \ref{thm:intro-example} is carried out in Section \ref{sec:app}.

\subsection*{Acknowledgments}

We would like to thank Arend Bayer, Gavril Farkas, Albrecht Klemm, Chunyi Li, Shengxuan Liu, Kenneth Ma, Emanuele Macr\`i, Rahul Pandharipande, Alexander Perry, Boris Pioline, Yongbin Ruan, Hao Max Sun, Richard Thomas, Yukinobu Toda, Chenyang Xu, and Xiaolei Zhao for many helpful discussions. We thank the anonymous referee for a careful reading as well as a list of useful suggestions that improved the exposition of the paper. This paper was completed while Z.L. was visiting the Institute for Theoretical Studies at ETH Z\"urich, the Mathematics Institute at the University of Warwick, and Imperial College London, whose hospitality he gratefully acknowledges. S.F. was supported by the Royal Society URF/R1/23119. N.K. and N.R. were supported by the Engineering and Physical Sciences Research Council EP/Y009959/1. Z.L. was supported by NSFC Grant 123B2002.

\section{Preliminaries}\label{sec:pre}

Throughout this paper, all schemes are defined over $\C$. In this section, we review the necessary background on Chern characters, slope-stability, and tilt-stability.

For a proper one-dimensional scheme $C$ over $\C$, its (arithmetic) \emph{genus} $g(C)$ is defined to be $1-\chi(\cO_C)$. A \emph{polarised scheme} $(X, H)$ is a scheme $X$ equipped with an ample divisor $H$. We denote the restriction $H|_Z$ of $H$ to a subscheme $Z\subset X$ by $H_Z$. The bounded derived category of coherent sheaves on $X$ is denoted by $\mathrm{D^b}(X)$. We write $h^i(-), \mathrm{hom}(-,-)$, and $\mathrm{ext}^i(-,-)$ for the dimensions of $H^i(-), \mathrm{Hom}(-,-)$, and $\mathrm{Ext}^i(-,-)$, respectively. For any $b\in \mathbb{R}$, we denote by $\lfloor b\rfloor$ the greatest integer that $\le b$.

\subsection{Chern characters}

Let $X$ be a quasi-projective scheme over $\C$. Write $\mathrm{K}_0(X)\coloneqq\mathrm{K}(\mathrm{D^b}(X))$ for the K-group of coherent sheaves and $\mathrm{K}^0(X)\coloneqq\mathrm{K}(\mathrm{Perf}(X))$ for the K-group of vector bundles. Write $\CH_k(X)$ for the Chow group of $k$-dimensional cycles in $X$, and set $$\CH_*(X)\coloneqq\bigoplus_{k} \CH_k(X) \quad \text{and} \quad\CH_*(X)_{\mathbb{F}}\coloneqq\CH_*(X)\otimes_{\Z}\mathbb{F}$$ for any field $\mathbb{F}$. As in \cite[Chapter 17]{fulton:intersection-theory}, we have the bivariant Chow rings $A^*(X)$ and $A^*(X)_{\mathbb{F}}\coloneqq A^*(X)\otimes_{\Z} \mathbb{F}$ which act on $\CH_*(X)$ and $\CH_*(X)_{\mathbb{F}}$, respectively, by the cap product $\cap$.

We use the notion of \emph{local complete intersection (lci)} morphisms as in \cite[\href{https://stacks.math.columbia.edu/tag/069F}{Tag 069F}]{stacks-project}. By \cite[\href{https://stacks.math.columbia.edu/tag/068E}{Tag 068E}]{stacks-project}, lci morphisms are stable under compositions and flat base change. According to \cite{avramov:lci}, a flat morphism $f\colon X\to Y$ of finite type between Noetherian schemes is lci if and only if the cotangent complex $\mathbb{L}_{X/Y}$ is perfect with Tor-amplitude in $[-1,0]$.

\begin{Def}
Let $X$ be a quasi-projective lci scheme. The \emph{virtual Todd class of $X$} is defined as
\[\td(X)\coloneqq\td(\mathbb{T}_{X/\mathbb{C}})\in A^*(X)_{\Q},\]
where $\td(\mathbb{T}_{X/\mathbb{C}})\in A^*(X)_{\Q}$ is the Todd class of the tangent complex $$\mathbb{T}_{X/\mathbb{C}}\coloneqq R\hom_X(\mathbb{L}_{X/\mathbb{C}},\cO_X).$$
\end{Def}

\noindent In this case, the cotangent complex $\mathbb{L}_{X/\mathbb{C}}$ is represented by a bounded complex of vector bundles, making $\td(X)\in A^*(X)_{\Q}$ well-defined.

It is clear that when $X$ is smooth, the class $\td(X)$ recovers the usual Todd class of $X$ defined by its tangent bundle. More generally, the class $[\mathbb{T}_{X/\mathbb{C}}]\in \mathrm{K}^0(X)$ coincides with the virtual tangent bundle in the sense of \cite[B.7.6]{fulton:intersection-theory}, so $\td(X)$ agrees with the virtual Todd class used in \cite{fulton:intersection-theory}.

Recall that for any quasi-projective scheme $X$ over $\mathbb{C}$, there is a Riemann--Roch homomorphism
\[\tau_X\colon \mathrm{K}_0(X)\to \CH_*(X)_{\Q}\]
satisfying many functorial properties as in the smooth case (cf.~\cite[Theorem 18.3]{fulton:intersection-theory}). Since we can always take the formal inverse of a non-zero class in $A^*(X)_{\Q}$ (cf.~\cite[\href{https://stacks.math.columbia.edu/tag/0ESY}{Tag 0ESY}]{stacks-project}), motivated by the Grothendieck--Riemann--Roch Theorem, we define the Chern characters as follows.

\begin{Def}
Let $X$ be a quasi-projective lci scheme. For any class $\gamma\in \mathrm{K}_0(X)$, we define the \emph{Chern characters} of $\gamma$ as
\[\ch(\gamma)\coloneqq (\td(X))^{-1}\cap \tau_X(\gamma) \in \CH_*(X)_{\Q}.\]
\end{Def}

\noindent We write the component of $\ch(\gamma)$ of dimension $\dim X-k$ by 
\[\ch_k(\gamma)\in \CH_{\dim X-k}(X)_{\Q}.\]
In particular, we obtain a group homomorphism
\[\ch\colon \mathrm{K}_0(X)\to \CH_*(X)_{\Q}.\]
If we write $i\colon \mathrm{K}^0(X)\to \mathrm{K}_0(X)$ for the natural homomorphism, then, by \cite[Corollary 18.3.1(b)]{fulton:intersection-theory}, we have
\begin{equation*}
\ch(i(\cO_X))=[X].
\end{equation*}
Therefore, from \cite[Corollary 18.3(2)]{fulton:intersection-theory}, for any class $\gamma\in \mathrm{K}^0(X)$, we get $$\ch(\gamma)\cap [X]=\ch(i(\gamma))\in \CH_*(X)_{\Q},$$
where $\ch(\gamma)\in A^*(X)_{\Q}$ is the usual bivariant Chern character. Thus, by abuse of notation, we do not distinguish the usual bivariant Chern characters $\ch(\gamma)$ and $\ch(i(\gamma))$ defined above.

The notion of Chern characters satisfies the following general properties.

\begin{Lem}\label{lem:property-ch}
Let $X$ and $Y$ be quasi-projective lci schemes.

\begin{enumerate}

    \item If $f\colon X\to Y$ is proper, then for any class $\gamma\in \mathrm{K}_0(X)$, we have
    \[f_*(\td(X)\cap \ch(\gamma))=\td(Y)\cap \ch(f_*\gamma).\]

    \item If $f\colon X\to Y$ is lci (e.g.~$f$ is flat), then for any class $\gamma\in \mathrm{K}_0(Y)$, we have
    \[\ch(f^*\gamma)=f^*\ch(\gamma).\]

    \item If $\alpha\in \mathrm{K}^0(X)$ and $\gamma\in \mathrm{K}_0(X)$, then
    \[\ch(\alpha\otimes \gamma)=\ch(\alpha)\cap \ch(\gamma).\]
Here $\ch(\alpha)\in A^*(X)_{\Q}$ and $\ch(\gamma), \ch(\alpha\otimes \gamma)\in \CH_*(X)_{\Q}$.

    \item If $V\subset X$ is an equidimensional closed subscheme, then
    \[\ch(\cO_V)=[V]+\text{ terms of lower dimensions}.\]
\end{enumerate}

\end{Lem}

\begin{proof}
These results follow from the definition of $\ch(-)$ and \cite[Theorem 18.3]{fulton:intersection-theory}. Note that in part (b), if $f$ is flat, then it is lci by \cite[\href{https://stacks.math.columbia.edu/tag/09RL}{Tag 09RL}]{stacks-project}, and we have
\[
  [\mathbb{T}_{X/Y}]+[f^*\mathbb{T}_{Y/\C}]=[\mathbb{T}_{X/\C}]\in \mathrm{K}^0(X)
\]
by the standard exact triangle of cotangent complexes.
\end{proof}

\begin{Rem}\label{rmk:ch}
We comment on the properties of $\ch(-)$.
\begin{itemize}

\item Part (a) of Lemma \ref{lem:property-ch} should be understood as the Grothendieck--Riemann--Roch theorem for bounded complexes of coherent sheaves on lci schemes. In particular, if $Y$ is a point, then
\[\chi(\gamma)=\td(X)\cap \ch(\gamma)\in \Z\]
for any $\gamma\in \mathrm{K}_0(X)$, which is the lci version of Hirzebruch--Riemann--Roch.

\item When $f\colon X\to Y$ is an inclusion of an effective Cartier divisor, Lemma \ref{lem:property-ch}(b) shows that $\ch(-)$ is well-behaved under restriction.

\item When $Y$ is integral and $f\colon Y_{\mathrm{reg}}\to Y$ is the inclusion of the smooth part, Lemma \ref{lem:property-ch}(b) implies $\ch_0(\gamma)=\rk(\gamma)[Y]$ for any class $\gamma \in \mathrm{K}_0(Y)$, where $\rk(-)$ denotes the rank. By abuse of notation, we do not distinguish between $\ch_0(\gamma)$ and $\rk(\gamma)$. Furthermore, if $Y$ is normal, then $\ch_1(\gamma)\in \CH_{\dim Y-1}(Y)$ is the closure of the Cartier divisor $\ch_1(f^*\gamma)$ on the smooth locus $Y_{\mathrm{reg}}$.

\end{itemize}

\end{Rem}

\subsection{Slope-stability}

Next, we review the theory of semistable sheaves. We refer to \cite{huybrechts:geometry-of-moduli-space-of-sheaves} for a detailed introduction. 

Given a projective lci variety $X$ over $\C$ and an ample divisor $H$, we define the $H$-slope of an object $E\in \mathrm{D^b}(X)$ by
\[\mu_H(E)\coloneqq\frac{\ch_1(E).H^{\dim X-1}}{\ch_0(E)H^{\dim X}}\]
when $\ch_0(E)\neq 0$ and set $+\infty$ otherwise. Here, $Z.H^{\dim X-k}$ is defined as
\[Z.H^{\dim X-k}\coloneqq \int_X c_1(\cO_X(H))^{\dim X-k}\cap Z\in \Q\]
for any $k$-dimensional cycle $Z\in \CH_k(X)_{\Q}$. Note that if $X$ is normal, then $\ch_1(E)$ is a Weil divisor obtained as the closure of the corresponding divisor on the smooth locus of $X$ (cf.~Remark \ref{rmk:ch}). Therefore, we have $\ch_1(E).H^{\dim X-1}\in \Z$ in this case.

A torsion-free sheaf $E$ is \emph{$H$-(semi)stable} if, for any non-zero proper subsheaf $F\subset E$, we have
\[\mu_H(F)(\le) \mu_H(E/F),\]
where $(\le)$ denotes $<$ for stability and $\le$ for semistability. Every coherent sheaf admits a Harder--Narasimhan (HN) filtration, whose first factor is its torsion part and the other factors are $H$-semistable sheaves. We denote by $\mu^+_H(-)$ and $\mu^-_H(-)$ the slopes of the first and the last factor in the HN filtration, respectively.

If $X$ is also smooth or a normal projective surface, we define the \textit{discriminants} of a sheaf $E$ as follows: 
\begin{align*}
    &\Delta(E)\coloneqq\ch_1(E)^2-2\ch_0(E)\ch_2(E), \\
    &\Delta_H(E)\coloneqq(\ch_1(E).H^{\dim X-1})^2-2\ch_0(E)H^{\dim X}(\ch_2(E).H^{\dim X-2}).
\end{align*}
Note that when $X$ is a normal projective surface, $\ch_1(E)^2\in \mathbb{Q}$ is defined as Mumford's intersection number of Weil divisors. If $X$ is smooth, then the classical Bogomolov--Gieseker (BG) inequality gives
\begin{equation}\label{eq-delta}
    \Delta_H(E)\geq H^{\dim X}(\Delta(E).H^{\dim X -2}) \geq 0
\end{equation}
for any $H$-semistable torsion-free sheaf $E$ on $X$. More generally, by \cite{langer:normal-surface,nuer:bg-inequality-singular-surface}, this still holds if $X$ is a normal projective surface\footnote{Indeed, a version of BG inequality holds for more general class of varieties. This allows us to generalize results in \cite{bayer:the-space-of-stability-conditions-on-abelian-threefolds,bayer:stability-conditions-in-families} to the singular setting in \cite{liu-mao:tilt-singular}.} with rational Gorenstein (hence lci) singularities.

\subsection{Tilt-stability}\label{subsec:tilt}

In the rest of this section, we assume $X$ is either a smooth projective variety or a normal projective surface with rational Gorenstein singularities, and we fix an ample divisor $H$ on $X$. Set $n\coloneqq\dim X\geq 2$.

For any $b \in \mathbb{R}$, we define the tilted heart
\begin{equation*}
\Coh^b(X)\ \coloneqq \ \big\{E^{-1} \xrightarrow{\,d\,} E^0 \ \colon\ \mu_H^{+}(\ker d) \leq b \,,\  \mu_H^{-}(\cok d) > b \big\} \subset\mathrm{D^b}(X),
\end{equation*}
which is the heart of a bounded t-structure on $\mathrm{D^b}(X)$ by \cite[Lemma 6.1]{bridgeland:K3-surfaces}.

For any object $E \in \Coh^b(X)$, we set
\begin{equation*}
\nu_{b,w}(E)\ \coloneqq\ \left\{\!\!\begin{array}{cc} \frac{\ch_2(E).H^{n-2} - w\ch_0(E)H^{n}}{\ch_1^{bH}(E).H^{n-1}}
 & \text{if }\ch_1^{bH}(E).H^{n-1}\ne0, \\
+\infty & \text{otherwise, } \end{array}\right.
\end{equation*}
where $\ch^{bH}(E)\coloneqq \exp(-bH).\ch(E)$ and $\ch^{bH}_k(E)$ is its codimension $k$ part.

\begin{Def}
Fix a pair $(b,w) \in U\coloneqq \{(b,w)\in \R^2\colon w>\frac{1}{2}b^2\}$. We say $E\in\mathrm{D^b}(X)$ is \emph{$\nu_{b,w}$-(semi)stable} if and only if
\begin{itemize}
\item $E[k]\in \Coh^b(X)$ for some $k\in\mathbb{Z}$, and
\item $\nu_{b,w}(F)\,(\le)\,\nu_{b,w}\big(E[k]/F\big)$ for all non-trivial subobjects $F\hookrightarrow E[k]$ in $\Coh^b(X)$.
\end{itemize}
\end{Def}

\noindent Analogous to $H$-stability of sheaves, for any $(b,w)\in U$ and non-zero object $E\in \Coh^b(X)$, there exists a filtration of $E$ consisting of subobjects of $E$ in $\Coh^b(X)$ such that each factor is $\nu_{b,w}$-semistable. We call this filtration the \emph{HN filtration of $E$ with respect to $\nu_{b,w}$}. For more details, we refer to \cite{bayer:bridgeland-stability-conditions-on-threefolds,bayer:the-space-of-stability-conditions-on-abelian-threefolds}.

By \cite[Theorem 3.5]{bayer:the-space-of-stability-conditions-on-abelian-threefolds}, if we plot the $(b,w)$-plane simultaneously with the image of the projection map
\begin{eqnarray*}
	\Pi\colon\ \mathrm{K}_0(X) \setminus \big\{v \colon \ch_0(v) = 0\big\}\! &\longrightarrow& \R^2, \\
	v &\ensuremath{\shortmid\joinrel\relbar\joinrel\rightarrow}& \!\!\bigg(\frac{\ch_1(v).H^{n-1}}{\ch_0(v)H^{n}}\,,\, \frac{\ch_2(v).H^{n-2}}{\ch_0(v)H^{n}}\bigg),
\end{eqnarray*}
then the image $\Pi(E)$ of a $\nu_{b,w}$-semistable object $E$ with $\ch_0(E)\ne 0$ is \emph{outside} $U$, i.e. $\Delta_H(E)\geq 0.$ Moreover, we have the following explicit description of a wall and chamber structure, explained in \cite[Proposition 4.1]{feyz:thomas-noether-loci} for instance\footnote{The wall and chamber structure in the singular case can be proven by an argument similar to \cite{bayer:the-space-of-stability-conditions-on-abelian-threefolds}. See \cite{liu-mao:tilt-singular} for a detailed proof}.

\begin{Lem}[\textbf{Wall and chamber structure}]\label{locally finite set of walls}
	Fix $v\in \mathrm{K}_0(X)$ with $\Delta_H(v)\ge0$ and $$(\ch_{0}(v)H^n, \ch_1(v).H^{n-1}, \ch_2(v).H^{n-2})\ne (0,0,0).$$ There exists a set of lines $\{\ell_i\}_{i \in I}$ in $\mathbb{R}^2$ such that the segments $\ell_i\cap U$ (called ``\emph{walls of instability}") are locally finite and satisfy 
	\begin{itemize}
	    \item[\emph{(}a\emph{)}] If $\ch_0(v)\ne0$, then all lines $\ell_i$ pass through $\Pi(v)$.
	    \item[\emph{(}b\emph{)}] If $\ch_0(v)=0$, then all lines $\ell_i$ are parallel of slope $\frac{\ch_2(v).H^{n-2}}{\ch_1(v).H^{n-1}}$.
	   		\item[\emph{(}c\emph{)}] The $\nu_{b,w}$-(semi)stability of any $E\in\mathrm{D^b}(X)$ of class $v$ is unchanged as $(b,w)$ varies within any connected component (called a ``\emph{chamber}") of  $U \setminus \bigcup_{i \in I}\ell_i$.
		\item[\emph{(}d\emph{)}] For any wall $\ell_i\cap U$, there is an integer $k_i$ and a map $f\colon F\to E[k_i]$ in $\mathrm{D^b}(X)$ such that
\begin{itemize}
\item for any $(b,w) \in \ell_i \cap U$, the objects $E[k_i],\,F$ lie in the heart $\Coh^{b}(X)$,
\item $E$ is $\nu_{b,w}$-semistable of class $v$ with $\nu_{b,w}(E)=\nu_{b,w}(F)=\,\mathrm{slope}\,(\ell_i)$ constant on the wall $\ell_i \cap U$, and
\item $f$ is an injection $F\hookrightarrow E[k_i] $ in $\Coh^{b}(X)$ which strictly destabilizes $E[k_i]$ for $(b,w)$ in one of the two chambers adjacent to the wall $\ell_i$.
\end{itemize} 
	\end{itemize}
\end{Lem}

Finally, we recall the notion of Brill--Noether stability. 

\begin{Def}
For an object $E\in \mathrm{D^b}(X)$, we say $E$ is \emph{Brill--Noether (BN) stable} if $E$ is $\nu_{b,w}$-stable for any $(b,w)$ in an open subset
\[\left\{(b,w)\colon b^2+w^2<\delta, w>\frac{1}{2}b^2\right\}\]
for some $\delta>0$. We say $E$ is \emph{BN-semistable} if $E$ is $\nu_{0,w}$-semistable for every $0<w \ll 1$. 
\end{Def}

\noindent By Lemma \ref{locally finite set of walls}, if $\ch_0(E)\neq 0$ and $\ch_2(E).H^{n-2}\neq 0$, then $E$ is BN-stable if and only if it is $\nu_{b,w}$-stable for some $(b,w) \in U$ proportional to $\Pi(E)$. If $E$ is BN-semistable, then $E$ is $\nu_{b,w}$-semistable for some $(b,w) \in U$ proportional to $\Pi(E)$. 
For any object $E\in \Coh^0(X)$, we define its Brill--Noether slope as
\[\nu_{BN}(E)\coloneqq\frac{\ch_2(E).H^{n-2}}{\ch_1(E).H^{n-1}}\]
when $\ch_1(E).H^{n-1}\neq 0$ and set $+\infty$ otherwise.

We need the following property of BN-stable objects.

\begin{Lem}\label{lem:vanish-hom-bn}
Let $(X, H)$ be a polarised smooth projective threefold with $K_X.H^2\leq 0$. For any object $E\in \Coh^0(X)$, we have
\[\chi(E)\leq \mathrm{hom}(\cO_X, E)+\mathrm{ext}^2(\cO_X, E).\]
Furthermore, if $K_X.H^2=K_X^2.H=0$ and $E$ is BN-stable with $\nu_{BN}(E)>0$, then
\[\chi(E)\leq \mathrm{hom}(\cO_X, E).\]
\end{Lem}

\begin{proof}
Since $\cO_X[1], E\in \Coh^0(X)$, we have
\[\mathrm{Hom}(\cO_X[1],E[-i])=0\]
for any $i>0$. As $K_X.H^2\leq 0$, we have $K_X[1]\in \Coh^0(X)$, which gives
\[\mathrm{Hom}(E,K_X[1-i])=\mathrm{Hom}(\cO_X,E[i+2])=0\]
for any $i>0$. These two vanishing results together imply
\[\chi(E)\leq \mathrm{hom}(\cO_X, E)+\mathrm{ext}^2(\cO_X, E).\]

When $K_X.H^2=K_X^2.H=0$, we know that $K_X[1]\in \Coh^b(X)$ is $\nu_{b,w}$-stable for any $b\geq 0$ and $w> \frac{1}{2}b^2$. From the BN-stability of $E$ and $\nu_{BN}(E)>0$, we can find $(b_0,w_0)\in U$ with $b_0>0$ sufficiently small such that $E\in \Coh^{b_0}(X)$ is $\nu_{b_0,w_0}$-stable and
\begin{equation}\label{eq-lem-sec-2}
    0<\frac{w_0}{b_0}<\nu_{BN}(E).
\end{equation}
Since \eqref{eq-lem-sec-2} implies
\[\nu_{b_0,w_0}(K_X[1])=\frac{w_0}{b_0}<\nu_{b_0,w_0}(E),\]
we get $\mathrm{Hom}(E,K_X[1])=\mathrm{Ext}^2(\cO_X, E)=0$ and the second inequality follows.
\end{proof}

The following helps us to control the slope of Jordan--H\"older factors of a semistable object.

\begin{Lem}\label{lem:JH-factor}
Fix $b\in \Z$ and $w_0>\frac{1}{2}b^2$. Let $E\in \Coh^b(X)$ be a strictly $\nu_{b,w_0}$-semistable object with $\ch_0(E)\neq 0$. Suppose that $\mu_H(E)>b$ (resp.~$\mu_H(E)<b$). Then there exists a Jordan--H\"older factor $F$ of $E$ with respect to $\nu_{b,w_0}$ such that $\mu_H(F)\in (b, \mu_H(E)]$ (resp.~$\mu_H(F)\in [\mu_H(E),b)$) and $\ch^{bH}_1(F).H^{n-1}<\ch_1^{bH}(E).H^{n-1}$.
\end{Lem}

\begin{proof}
By replacing $E$ with $E(-bH)$, we may assume that $b=0$. Since $E\in \Coh^0(X)$ and $\mu_H(E)\neq 0$, we have $\ch_1(E).H^{n-1}>0$. Therefore, the Jordan--H\"older factors $F_1,\cdots, F_m$ of $E$ with respect to $\nu_{0,w_0}$ satisfy $$0<\ch_1(F_k).H^{n-1}<\ch_1(E).H^{n-1}$$ for any $1\leq k\leq m$.

From the assumption, we get $\mathrm{sign}(\mu_H(E))\ch_0(E)>0$, where $\mathrm{sign}(a)\coloneqq 1$ if $a>0$ and $\mathrm{sign}(a)\coloneqq -1$ if $a<0$. After reordering, we may assume that $\mathrm{sign}(\mu_H(E))\ch_0(F_k)>0$ if and only if $1\leq k\leq m_0$, where $m_0\leq m$. Thus, we see
\[\sum^{m_0}_{k=1} \mathrm{sign}(\mu_H(E))\ch_0(F_k)\geq \mathrm{sign}(\mu_H(E))\ch_0(E)\quad\] 
and
\[\sum^{m_0}_{k=1} \ch_1(F_k).H^{n-1}\leq \ch_1(E).H^{n-1}.\]
In particular, we obtain
\[\frac{\sum^{m_0}_{k=1} \ch_1(F_k).H^{n-1}}{\sum^{m_0}_{k=1} \mathrm{sign}(\mu_H(E))\ch_0(F_k)H^n}\leq \mathrm{sign}(\mu_H(E))\mu_H(E),\]
and the result follows.
\end{proof}

\subsection{Bridgeland stability conditions} In \cite[Conjecture 2.7]{bayer:bridgeland-stability-conditions-on-threefolds}, Bayer–Macr\`i–Toda proposed a Bogomolov--Gieseker-type inequality involving $\ch_3$ for $\nu_{b,w}$-semistable objects. It was subsequently extended in \cite{bayer:the-space-of-stability-conditions-on-abelian-threefolds,macri:stability-fano-3fold, bayer-macri:icm-report}. One of the formulations is as follows:

\begin{Conj} [{$\Gamma$-generalized BMT inequality}]\label{conj:bmt}
Let $(X,H)$ be a polarised smooth projective threefold. There exists a $1$-cycle $\Gamma\in \CH_1(X)_{\mathbb{R}}$ with $\Gamma. H \ge 0$ such that, for any
\[
w>\tfrac12 b^2+\tfrac12\big(b-\lfloor b\rfloor\big)\big(1-b+\lfloor b\rfloor\big)
\]
and any $\nu_{b,w}$-semistable object $E$, we have
\begin{align*}
0 \leq Q^{\Gamma}_{b,w}(E)
&\coloneqq (2w-b^2)\!\left(\Delta_H(E)+3\,\frac{\Gamma. H}{H^3}\big(\ch_0^{bH}(E)H^3\big)^2\right) \\
&\quad +\,2\big(\ch_2^{bH}(E). H\big)\!\left(2\,\ch_2^{bH}(E). H-3\,(\Gamma. H)\,\ch_0^{bH}(E)\right) \\
&\quad -\,6\big(\ch_1^{bH}(E). H^2\big)\!\left(\ch_3^{bH}(E)-\Gamma. \ch_1^{bH}(E)\right). 
\end{align*}
\end{Conj}

Analogous to \cite[Theorem 8.2]{bayer:the-space-of-stability-conditions-on-abelian-threefolds}, Conjecture \ref{conj:bmt} enables us to construct an open subset in the space of geometric Bridgeland stability conditions on $\mathrm{D^b}(X)$ via the double-tilting construction. We refer to \cite[Section 8]{bayer:the-space-of-stability-conditions-on-abelian-threefolds} and \cite[Section 7]{koseki:double-triple-solids} for a detailed treatment.

Note that by rearranging, we find that $Q^{\Gamma}_{b,w}(E)$ is actually linear in $(b,w)$:
\begin{align}\label{Q-simple}
    \frac{1}{2}Q^{\Gamma}_{b,w}(E) \=  & w\left(C_1^2-2C_0C_2 +3 \frac{\Gamma.H}{H^3} C_0^2\right)+b \big(3C_0C_3-C_1C_2 -3C_0\Gamma.\ch_1(E)\big)\\
    & +2C_2^2-3C_1C_3 -3 \frac{\Gamma.H}{H^3}C_0C_2 +3C_1\Gamma.\ch_1(E)  , \nonumber
\end{align}
where $C_i\coloneqq \ch_i(E).H^{3-i}$.

\section{On Bayer--Macr\`i--Toda conjecture}\label{sec:bmt}

In this section, we prove our main theorem, which reduces Conjecture \ref{conj:bmt} to verifying a Brill--Noether-type inequality for curves in terms of the invariant $\bn_C$, which is defined in Definition \ref{def:bnc}.

Before stating the theorem, we propose a potential improvement of the classical BG inequality, which plays a key role in our argument. For a polarised projective lci scheme $(X,H)$ of pure dimension $n\ge 2$ and a real number $\epsilon>0$, we consider:

\begin{minipage}{14.2cm}\vspace{2mm}
\begin{enumerate}[label=$\mathbf{BG_n(\epsilon)}$]
\item\label{BGn} \emph{~For any $H$-stable sheaf $E$ on $X$ with $\mu_H(E)\in (0,\epsilon]$, we have
\begin{equation}\label{eq:stBG}
    \ch_2(E).H^{n-2}<-\frac{1}{2}\ch_1(E).H^{n-1}.
\end{equation}
}
\end{enumerate}
\end{minipage}

\begin{Thm}\label{thm:main-criterion}
Let $(X, H)$ be a polarised smooth projective threefold with $K_X$ numerically trivial and $H^1(\cO_X)=0$. Given a surface $S\in |H|$ and a curve $C\in |H_S|$. Assume either

\begin{enumerate}
    \item $S$ and $C$ are both smooth with
    \begin{equation}\label{eq:thm-bn-smooth}
        \bn_C< \chi(\cO_X(H)),
    \end{equation}
    or

    \item $S$ has rational singularities and $C$ is integral with
    \begin{equation*}
        \bn_C< \chi(\cO_X(H))-1.
    \end{equation*}
\end{enumerate}

Then \emph{\ref{conj:bmt}} holds for $(X,H)$ and some $1$-cycle $\Gamma\in \CH_1(X)_{\R}$ with $\Gamma.H\geq 0$.
\end{Thm}
\begin{proof}
By \cite[Theorem 3.2]{chunyi:stability-condition-quintic-threefold} or \cite[Theorem 2.3]{koseki:double-triple-solids}, it suffices to show that there exists a $1$-cycle $\Gamma\in \CH_1(X)_{\mathbb{R}}$ with $\Gamma. H\ge 0$ such that
\begin{equation}\label{Q}
Q^{\Gamma}_{0,0}(E)\ge 0    
\end{equation}
for every BN-stable object $E\in \Coh^{0}(X)$ with $\nu_{BN}(E)\in [0,\tfrac12]$. Then the proof is divided into the following three steps:

\medskip

\noindent \textbf{Step I.}
By Theorem \ref{thm:ch2}, to verify \eqref{Q}, it suffices to verify \ref{BG3} for $(X,H)$ for some $\epsilon>0$.

\medskip

\noindent \textbf{Step II.}
By Proposition \ref{prop:surface-to-3fold-1}, \ref{BG3} for $(X,H)$ reduces to
\hyperref[BGn]{\ensuremath{\mathbf{BG_2(\delta)}}} for $(S,H_S)$, where $S\in|H|$ and $\epsilon=\frac{\delta}{2+3\delta}$.

\medskip

\noindent \textbf{Step III.}
Finally, after noting $\chi(\cO_S)=\chi(\cO_X(H))$, Propositions \ref{prop:curve-to-surface-smooth} and \ref{prop:curve-to-surface} reduce
\hyperref[BGn]{\ensuremath{\mathbf{BG_2(\delta)}}} for $(S,H_S)$ to the required bounds in (a) and (b), respectively.
\end{proof}

\begin{Rem}
In Theorem \ref{thm:main-criterion}, the restriction of $\cO_X(H)$ to the curve $C$ has degree $g(C)-1$ and
\[
h^0\bigl(\cO_C(H_C)\bigr)=\chi\bigl(\cO_X(H)\bigr)-2,
\]
so we obtain
\begin{equation*}
\chi\bigl(\cO_X(H)\bigr)-2 \le \bn_C.
\end{equation*}
Although in some situations $\mathcal{O}_C(H_C)$ computes $\bn_C$ (see Remark \ref{rem-equality}), this is not always the case; cf.\ Theorem \ref{thm:cicy-in-fano} and Example~\ref{ex:pathology}.
\end{Rem}

\subsection{Step I: from $\ch_2$ to $\ch_3$}

The first step is to reduce the inequality for $\ch_3$ to certain inequalities for $\ch_2$. 

We begin by observing that \ref{BGn} automatically extends to a class of tilt-semistable objects.

\begin{Lem}\label{lem:st-to-tilt-st}
Let $(X, H)$ be a polarised smooth projective variety of dimension $n \geq 2$. If \ref{BGn} holds for $(X,H)$, then for any $w_0>0$ and $\nu_{0,w_0}$-semistable object $E\in \Coh^0(X)$ with $|\mu_H(E)| \in (0,\epsilon]$, we have
\begin{equation}\label{eq:stBG-tilt}
    \frac{\ch_2(E).H^{n-2}}{\ch_0(E)H^n} < -\frac{1}{2} \left|\mu_H(E)\right|.
\end{equation}
\end{Lem}

\begin{proof}
If $E$ is $\nu_{0,w}$-semistable for any $w\gg 0$, then by \cite[Lemma 2.7]{bayer:the-space-of-stability-conditions-on-abelian-threefolds}, either $E$ is an $H$-semistable sheaf, or $\cH^{-1}(E)$ is an $H$-semistable sheaf with $\dim \mathrm{Supp}(\cH^0(E))\leq n-2$. In both cases, the statement holds for $E$ by \ref{BGn}.

Now, we do induction on $\ch_1(E).H^{n-1}\in \Z_{> 0}$. If $\ch_1(E).H^{n-1}$ is minimal, then $E$ is $\nu_{0,w}$-semistable for any $w>0$ and the result follows from the above paragraph. Next, assume that $\ch_1(E).H^{n-1}$ is not minimal, and the result holds for any $w_0>0$ and $\nu_{0,w_0}$-semistable object $F\in \Coh^0(X)$ with $|\mu_H(F)| \in (0, \epsilon]$ and $\ch_1(F).H^{n-1}<\ch_1(E).H^{n-1}$. If $E$ violates \eqref{eq:stBG-tilt}, then it is not $\nu_{0,w}$-semistable for $w\gg 0$. In this case, $E$ is strictly $\nu_{0,w'}$-semistable for some $w'>0$. Let $\ell$ be the line segment connecting $(0,w')$ and $\Pi(E)$. By Lemma \ref{lem:JH-factor}, there exists a Jordan--H\"older factor $F$ of $E$ with respect to $\nu_{0,w'}$ such that $\Pi(F)$ lies on $\ell\setminus \{(0, w')\}$. Since $E$ violates \eqref{eq:stBG-tilt}, the line segment $\ell$ lies on or above the curve $w=-\frac{1}{2}|b|$. Hence, $F$ also violates \eqref{eq:stBG-tilt}, contradicting the induction hypothesis.
\end{proof}

Now, we can reduce \ref{conj:bmt} to \ref{BG3}.

\begin{Thm} \label{thm:ch2}
Let $(X, H)$ be a polarised smooth projective threefold with $K_X$ numerically trivial. Suppose \ref{BG3} holds for $(X,H)$ and some $\epsilon>0$. Then there exists a $1$-cycle $\Gamma (\epsilon)\in \CH_1(X)_{\R}$ such that 
$\Gamma(\epsilon).H \geq 0$ and 
$$Q^{\Gamma(\epsilon)}_{0,0}(E) \geq 0$$ for any BN-stable object $E \in \Coh^0(X)$ with $\nu_{\text{BN}}(E) \in [0, \frac{1}{2}]$. 
Explicitly, we can take 
\[\Gamma(\epsilon)\coloneqq\gamma H^2-\td_{2}(X), \quad 
\gamma \geq \max\left\{\frac{4}{H^3\epsilon}, \frac{\td_2(X).H}{H^3} \right\}. \]
In particular, \emph{\hyperref[conj:bmt]{\ensuremath{\Gamma(\epsilon)}-\textbf{BMT}}
} holds for $(X,H)$.
\end{Thm}

\begin{proof}
We consider the universal extension 
\[0\to E \to \widetilde{E} \to \Hom(\cO_X, E) \otimes \cO_X[1] \to 0, \]
which is an exact sequence in $\Coh^0(X)$. First, assume that $\nu_{BN}(E) \neq 0$. Then the BN-stability of $E$ imposes that $\widetilde{E}$ is BN-semistable (see e.g.~\cite[Lemma 2.12]{chunyi:stability-condition-quintic-threefold}) and $\nu_{BN}(\widetilde{E})=\nu_{BN}(E)$. 

From the assumption \ref{BG3} and Lemma \ref{lem:st-to-tilt-st}, it follows that $\mu_H(E), \mu_H(\widetilde{E}) \notin [-\epsilon, \epsilon]$, which gives
\begin{equation} \label{eq:randa}
\ch_0(E)H^3 \leq \frac{1}{\epsilon}\ch_1(E).H^2
\end{equation}
and 
\begin{equation} \label{eq:bound-hom}
-\frac{\ch_1(E).H^2}{H^3 \epsilon} \leq  \ch_0(\widetilde{E}) = \ch_0(E) - \mathrm{hom}(\cO_X, E).
\end{equation}

On the other hand, the Hirzebruch--Riemann--Roch theorem and Lemma \ref{lem:vanish-hom-bn} imply 
\begin{equation} \label{eq:RR-BN}
\ch_3(E)+\td_{2}(X).\ch_1(E)=\chi(E) \leq \mathrm{hom}(\cO_X, E). 
\end{equation}
Combining the inequalities \eqref{eq:bound-hom} and \eqref{eq:RR-BN}, we get 
\begin{equation} \label{eq:ch3ineq}
\ch_3(E) \leq \ch_0(E)+\frac{1}{H^3\epsilon}\ch_1(E).H^2-\td_{2}(X).\ch_1(E). 
\end{equation}
Then by \eqref{Q-simple} via writing $C_i\coloneqq \ch_i(E).H^{3-i}$ and $\Gamma(\epsilon) = \gamma H^2 -\td_2(X)$, we get the required lower bound 
\begin{align*}
\frac{1}{2}Q^{\Gamma(\epsilon)}_{0,0}(E) = \ & 2C_2^2-3 \left(\gamma -\frac{ \td_2(X).H}{H^3} \right)C_0C_2 -3C_1C_3 +3C_1(\gamma H^2 -\td_2(X)).\ch_1(E) \\ 
\overset{\eqref{eq:ch3ineq}}{\geq} \ & 2C_2^2 -3 \left(\gamma -\frac{ \td_2(X).H}{H^3} \right)C_0C_2+3\gamma C_1^2 - 3C_1\left(\frac{C_0}{H^3} +\frac{C_1}{H^3\epsilon}\right) \\
\overset{\text{(BG)}}{\geq} \ & 2C_2^2-3 \left(\gamma -\frac{ \td_2(X).H}{H^3} \right)\frac{C_1^2}{2}+3\gamma C_1^2 - 3C_1\left(\frac{C_0}{H^3} +\frac{C_1}{H^3\epsilon}\right)\\
\overset{\eqref{eq:randa}}{\geq} \ & 2C_2^2 + \frac{3}{2}C_1^2 \left(\gamma + \frac{\td_2(X).H}{H^3}   - \frac{4}{H^3\epsilon} \right)\\
\geq \ & 0. 
\end{align*}
Here, (BG) stands for the classical BG inequality \eqref{eq-delta}, and we used
the assumption on $\gamma$ and Miyaoka's inequality $\td_{2}(X).H \geq 0$ in the last inequality.

It remains to consider the case $\nu_{BN}(E)=0$, i.e., $\ch_2(E).H=0$. 
In this case, the inequality $Q^{\Gamma(\epsilon)}_{0,0}(E) \geq 0$ is equivalent to 
$\ch_3(E) \leq \Gamma(\epsilon).\ch_1(E)$. 
Following the proof of \cite[Proposition 3.3]{chunyi:stability-condition-quintic-threefold}, for any $0 < \delta \ll 1$, there exists a filtration of $\widetilde{E}$ such that each factor $E_i$ is $\nu_{0, \alpha_i}$-semistable for some $\alpha_i>0$, and satisfies $\nu_{BN}(E_i) < \delta$. 
Applying Lemma \ref{lem:st-to-tilt-st} to each $E_i$, we have 
$\mu_H(E_i) \notin [-\epsilon, 0]$, which implies $\ch_1(E_i).H^2>-\epsilon\ch_0(E_i)H^3$.
Thus, by taking the sum over all factors, we get $\ch_1(\widetilde{E}).H^2>-\epsilon\ch_0(\widetilde{E})H^3$ and so inequality \eqref{eq:bound-hom} still holds in this case. 
Similar to the arguments in \cite[Proposition 3.3]{chunyi:stability-condition-quintic-threefold}, using the derived dual and \cite[Proposition 5.1.3(b)]{bayer:bridgeland-stability-conditions-on-threefolds}, we also have 
\[
\mathrm{ext}^2(\cO_X, E) \leq \frac{\ch_1(E).H^2}{H^3\epsilon}-\ch_0(E). 
\]
Combining these inequalities with the Hirzebruch--Riemann--Roch theorem and Lemma \ref{lem:vanish-hom-bn}, we get 
\begin{align*}
\ch_3(E)+\td_{2}(X).\ch_1(E)=\chi(E)
&\leq \mathrm{hom}(\cO_X, E)+\mathrm{ext}^2(\cO_X, E) \\
&\leq \frac{2\ch_1(E).H^2}{H^3\epsilon}. 
\end{align*}
By our choice of $\Gamma(\epsilon)$, we obtain 
$\ch_3(E) \leq \Gamma(\epsilon).\ch_1(E)$ as required. 
\end{proof}

Therefore, in the following, we only need to investigate \ref{BG3} and its variants. Before doing this, we make some comments on \ref{BG3}.

\begin{Rem}\label{rmk:basepoint-free}
Note if \ref{BGn} holds for $(X, mH)$ with $m>0$, then \hyperref[BGn]{\ensuremath{\mathbf{BG_n}(m\epsilon)}} holds for $(X, H)$. In some situations, it may be useful to rescale the polarisation $H$ to find a suitable surface $S\in |H|$; see for example, Theorem \ref{thm:hypergeo}.
\end{Rem}

\begin{Rem}\label{rmk-etale}
If $\pi\colon Y\to X$ is a finite \'etale morphism, then \ref{BGn} holds for $(X, H)$ if and only if \ref{BGn} holds for $(Y,\pi^*H)$. Indeed, this follows from the fact that slope-semistability is preserved under taking pullback and pushforward along $\pi$, together with the following calculations using Grothendieck--Riemann--Roch: 
\[\ch_i(\pi^*E).(\pi^*H)^{n-i}=\deg(\pi)\ch_i(E).H^{n-i}\]
and
\[\ch_i(\pi_*F).H^{n-i}=\ch_i(F).(\pi^*H)^{n-i}\]
hold for any $i$ and any sheaf $E$ and $F$ on $X$ and $Y$, respectively.
\end{Rem}

In practice, Remark \ref{rmk-etale} enables us to establish \ref{BG3} for certain non-simply connected Calabi--Yau threefolds (cf.~Example \ref{ex:quotient}). On the other hand, the following example indicates that we should restrict ourselves to Calabi--Yau threefolds with finite fundamental groups when considering \ref{BG3}.

\begin{Ex} 
Let $X$ be a Calabi--Yau threefold with infinite fundamental group. 
We claim that \ref{BG3} fails on $(X,H)$ for any $\epsilon>0$ and any ample divisor $H$ on $X$. Indeed, using the Beauville--Bogomolov decomposition theorem \cite{beauville:trivial-c1}, such $X$ admits a finite {\'e}tale Galois covering $\pi\colon Y\to X$ such that $Y$ is an abelian threefold or $Y=S\times C$ for a K3 surface $S$ and an elliptic curve $C$. By Remark \ref{rmk-etale}, we only need to show the failure of \ref{BG3} for $Y$ and an arbitrary ample divisor $D$ on $Y$.

When $Y$ is an abelian threefold, by \cite{mukai:semi-homo}, we know that for any rational number $\mu \in \Q$, there exists a $D$-stable vector bundle $E$ on $Y$ with $\mu_{D}(E)=\mu$ and $\Delta_{D}(E)=0$. This invalidates \ref{BG3} for $(Y, D)$ and any $\epsilon>0$.

When $Y=S\times C$, we consider the morphism $[m]\coloneqq \id_S \times \underline{m} \colon Y\to Y$, where $\underline{m} \colon C \to C$ is the multiplication map on $C$ by $m\in \Z_{>0}$. For the pullback $L$ of a line bundle of degree $1$ on $C$, the vector bundle $[m]_*L$ is $\mu_D$-semistable since $[m]$ is \'etale. Furthermore, we have $$\lim_{m\to \infty}\mu_D([m]_*L)=0,$$ $\mu_D([m]_*L)>0$, and $\ch_2([m]_*L)=0$. This also invalidates \ref{BG3} for $(Y, D)$ and any $\epsilon > 0$.

It is worth mentioning that a stronger version of \ref{conj:bmt} (with $\Gamma=0$) is proved by \cite{bayer:the-space-of-stability-conditions-on-abelian-threefolds} when $X$ is covered by an abelian threefold. On the other hand, if $X$ is covered by $S\times C$, the existence of Bridgeland stability conditions on $X$ can be deduced from \cite{liu:stab-on-product} and \cite[Proposition 4.12]{perry:stab-on-quotient-product}.
Note, however, that the stability conditions constructed in \textit{loc. cit.} are not geometric in general (cf.~\cite[Remark 1.3]{perry:stab-on-quotient-product}).
\end{Ex}

\subsection{Step II: from surfaces to threefolds}

Next, we apply restriction techniques to further reduce \hyperref[BGn]{\ensuremath{\mathbf{BG_3(\epsilon)}
}} to \hyperref[BGn]{\ensuremath{\mathbf{BG_2(\delta)}}}
.

\begin{Lem} \label{lem:restriction}
Let $(X, H)$ be a polarized smooth projective variety of dimension $n \geq 2$, or a polarized normal projective surface with rational Gorenstein singularities, such that \ref{BGn} fails for $(X, H)$ and some $0<\epsilon<1/3$. Then there exists a reflexive $H$-stable sheaf $E$ on $X$ with
\begin{equation*}
    0 < \mu_H(E) \leq \frac{2\epsilon}{1 - \epsilon} \quad \text{and} \quad 
    \ch_2(E). H^{n-2} \geq -\frac{1}{2} \ch_1(E). H^{n-1},
\end{equation*}
such that for any integral divisor $D \in |H|$, each HN factor $F_i$ of $E|_D$ (with respect to $H_D$-stability) satisfies
\[
0 \leq \mu_{H_D}(F_i) \leq \frac{2\epsilon}{1 - 3\epsilon}.
\]
\end{Lem}
\begin{proof}
    Consider the continuous function $f_\epsilon \colon \R_{\geq 0} \to \R$ with 
\[
f_\epsilon(x) \coloneqq \begin{cases}
-\frac{1}{2} x & \text{if} \ 
x \in [0, \epsilon], \\
\frac{1+\epsilon}{2(1-\epsilon)}x-\frac{\epsilon}{1-\epsilon} & \text{if}\ x \in [\epsilon, \frac{2\epsilon}{1-\epsilon}],\\
\frac{1}{2} x^2 & \text{if} \ 
x \geq \frac{2\epsilon}{1-\epsilon} .
\end{cases}
\]

\begin{figure}[H]
\centering

\begin{tikzpicture}[x=7cm,y=7cm,>=Stealth,line cap=round,line join=round]

\def\eps{0.10}

\pgfmathsetmacro{\Ex}{2*\eps/(1-\eps)}
\pgfmathsetmacro{\Ey}{0.5*\Ex*\Ex}

\coordinate (O) at (0,0);
\coordinate (A) at (\eps,{-0.5*\eps});
\coordinate (D) at (1,0.5);
\coordinate (E) at (\Ex,\Ey);

\draw[->] (-0.15,0) -- (1.30,0);
\draw[->] (0,-0.15) -- (0,0.60);

\draw[very thick,black!70,domain=-0.15:1.25,samples=200,smooth]
  plot (\x,{0.5*\x*\x});

\draw[very thick,blue] (O) -- (A) -- (D);

\foreach \P in {O,A,E,D}{
  \filldraw[fill=black!35,draw=black,thick] (\P) circle (1.6pt);
}


\draw[color=black] (0.1,-0.1) node {$\bigl(\epsilon,-\tfrac12\epsilon\bigr)$};


\draw[color=black] (0.16,0.13) node {$\left(\frac{2\epsilon}{1-\epsilon},\frac{2\epsilon^2}{(1-\epsilon)^2}\right)$};
\draw[color=black] (0.95,0.55) node {$\bigl(1,\tfrac12\bigr)$};
\draw[color=black] (1.12,0.76) node {$y=\frac{x^2}{2}$};

\end{tikzpicture}

\caption*{Figure. The blue graph represents the piecewise linear part of $f_{\epsilon}(x)$. The second line segment connects points $(\epsilon, -\frac{1}{2}\epsilon)$ and $(1, \frac{1}{2})$.}
\end{figure}

Let $\mathcal{S}_{\epsilon}$ be the set of all objects $E \in \mathrm{D}^{\mathrm{b}}(X)$ that are 
$\nu_{0,w}$–semistable for some $w>0$ or $\nu_{1,w'}$–semistable for some $w'>\tfrac{1}{2}$,  
satisfying  
\[
0 < \mu_H(E) \leq \frac{2\epsilon}{1-\epsilon}
\]
and
\begin{equation*}
\frac{\ch_2(E).H^{n-2}}{\ch_0(E)H^{n}} 
\;\geq\;
f_{\epsilon}\!\left(
\frac{\ch_1(E).H^{n-1}}{\ch_0(E)H^{n}}
\right)\geq -\frac{1}{2}\left(
\frac{\ch_1(E).H^{n-1}}{\ch_0(E)H^{n}}
\right).
\end{equation*}
Note that $\mathcal{S}_{\epsilon}$ is non-empty since \ref{BGn} fails for $(X, H)$ by our assumption.

Now, we pick an object $E \in \mathcal{S}_{\epsilon}$ with minimal $\Delta_H$ among all objects in $\mathcal{S}_{\epsilon}$. We first show that we can choose $E$ so that it is a reflexive $H$-stable sheaf and both $E$ and $E(-H)[1]$ are $\nu_{0, w}$-stable for any $w >0$. Indeed, we may assume that $E\in \Coh^0(X)$ is $\nu_{0,w}$-semistable for some $w>0$, as the argument for another case is similar. If $E$ is strictly $\nu_{0, w}$-semistable for some $w>0$, then by Lemma \ref{lem:JH-factor}, there exists a Jordan--H\"older factor $E'$ of $E$ with respect to $\nu_{0, w}$ such that $0<\mu_H(E')\leq \mu_H(E)$. Since the line segment connecting any point $(b,f_{\epsilon}(b))$ with $b\in [0,1]$ and $(0,0)$ is above or on the graph of $f_{\epsilon}$, $E'$ is also in $\mathcal{S}_{\epsilon}$. But this contradicts the minimality of $\Delta_H(E)$ as $\Delta_H(E')<\Delta_H(E)$ by \cite[Corollary 3.10]{bayer:the-space-of-stability-conditions-on-abelian-threefolds}. Therefore, $E$ is $\nu_{0,w}$-stable for any $w>0$ and so it is an $H$-stable sheaf by \cite[Lemma 2.7(c)]{bayer:the-space-of-stability-conditions-on-abelian-threefolds}.
As its double dual $E^{\vee \vee}$ is also $H$-stable and $E^{\vee \vee}\in \mathcal{S}_{\epsilon}$, with $\mu_H(E)=\mu_H(E^{\vee \vee})$ and $\Delta_H(E^{\vee \vee})\leq \Delta_H(E)$, we may assume that $E$ is reflexive. Hence, $E[1]\in \Coh^1(X)$ is $\nu_{1,w'}$-stable for $w'\gg \frac{1}{2}$ by \cite[Lemma 2.7(c)]{bayer:the-space-of-stability-conditions-on-abelian-threefolds}. Since the line segment connecting any point $(b,f_{\epsilon}(b))$ with $b\in [0,1]$ and $(1,\frac{1}{2})$ is above or on the graph of $f_{\epsilon}$, using Lemma \ref{lem:JH-factor} and \cite[Corollary 3.10]{bayer:the-space-of-stability-conditions-on-abelian-threefolds} again as in the previous case, we also know that $E[1]$ is $\nu_{1,w'}$-stable for any $w'>\frac{1}{2}$, which implies the desired $\nu_{0,w}$-stability of $E(-H)[1]$ for any $w>0$.

Now, we consider the following exact sequence in $\Coh^0(X)$: 
\begin{equation} \label{eq:restriction-1}
0 \to E \to E|_D \to E(-H)[1] \to 0. 
\end{equation}
Let $(0, w_0)$ be the intersection point of the line $b=0$ and the line passing through $\Pi(E)$ and $\Pi(E(-H)[1])$. 
If $w_0$ is positive, then \cite[Corollary 4.1]{feyz:effective-restriction-theorem} implies that $E|_D$ is $H_D$-stable with $\mu_{H_D}(E|_D)=\mu_H(E) \in (0, \frac{2\epsilon}{1-\epsilon}]$.

It remains to consider the case when $w_0 \leq 0$. In this case, we have $\nu_{BN}(E)\leq \nu_{BN}(E(-H)[1])$. If $F_i$ is an HN factor of $E|_D$ with respect to $\mu_{H_D}$, from \eqref{eq:restriction-1} and the $\nu_{0,w}$-stability of $E$ and $E(-H)[1]$, we get 
$$-\frac{1}{2} \leq \nu_{BN}(E)  \leq \nu_{BN}(\iota_*F_i) \leq \nu_{BN}(E(-H)[1]).$$
And by Lemma \ref{lem:property-ch}(a), we see that 
\[\nu_{BN}(\iota_*F_i)=\mu_{H_D}(F_i)-\frac{1}{2}, \]
where $\iota \colon D \hookrightarrow X$ denotes the inclusion. We also have
\begin{equation*}
\nu_{BN}(E(-H)[1]) +\frac{1}{2} 
= \frac{1- 2\frac{\ch_2(E).H^{n-2}}{\ch_0(E)H^{n}} }{2(1-\mu_H(E))}- \frac{1}{2} \leq \frac{1+\mu_H(E)}{2(1-\mu_H(E))}-\frac{1}{2} .
\end{equation*}
Recall $\mu_H(E) \in (0, 2\epsilon/(1-\epsilon)]$, which implies 
\[\frac{1+\mu_H(E)}{2(1-\mu_H(E))}-\frac{1}{2}
=\frac{\mu_H(E)}{1-\mu_H(E)}
\leq \frac{2\epsilon/(1-\epsilon)}{1-2\epsilon/(1-\epsilon)}
=\frac{2\epsilon}{1-3\epsilon}. \]
Putting these inequalities together, we obtain the claim. 
\end{proof}

As an immediate corollary, we obtain the desired dimensional reduction for \ref{BGn}.

\begin{Prop} \label{prop:surface-to-3fold-1}
Let $(X, H)$ be a polarised smooth projective variety of dimension $n \geq 3$, and $D \in |H|$ be an integral divisor. Suppose that \hyperref[BGn]{\ensuremath{\mathbf{BG_{n-1}(\delta)}}} holds for $(D, H_D)$ and some $\delta >0$, and $\ch_2(F).H^{n-3}_D\leq 0$ for any $H_D$-stable sheaf $F$ with $\mu_{H_D}(F)=0$. Then \ref{BGn} holds for $(X, H)$ and $\epsilon \coloneqq \frac{\delta}{2+3\delta}.$
\end{Prop}
\begin{proof}
 Suppose \ref{BGn} fails for $(X, H)$. Then, by Lemma \ref{lem:restriction}, there exists a reflexive $H$-stable sheaf $E$ of slope $\mu_H(E)\in (0,2\epsilon/(1-\epsilon)]$ such that 
\begin{equation}\label{v}
    \ch_2(E) . H^{n-2} \geq -\frac{1}{2}\, \ch_1(E) . H^{n-1}.
\end{equation}
Moreover, each HN factor $F_i$ of $E|_D$ satisfies
\[0 \leq \mu_{H_D}(F_i) \leq \frac{2\epsilon}{1 - 3\epsilon} = \delta.\]
Hence, by our assumption on $D$, summing over all factors yields
\[\ch_2(E|_D) . H_D^{\,n-3} < -\frac{1}{2}\, \ch_1(E|_D) . H_D^{\,n-2},\]
which contradicts \eqref{v}.
\end{proof}

\subsection{Step III: from curves to surfaces}

The final step is to investigate a two-dimensional analogue of Proposition \ref{prop:surface-to-3fold-1}. In this situation, the correct formulation of \hyperref[BGn]{\ensuremath{\mathbf{BG_{n}(\epsilon)}}} for a curve $C$ should be a bound on $\bn_C$ as defined in Definition \ref{def:bnc}.

Recall that for any coherent sheaf $E$ on an integral projective curve $C$ of (arithmetic) genus $g = h^1(\cO_C)$, we can define the degree and the slope of $E$ as
\[\deg(E)\coloneqq\chi(E)+\rk(E)(g-1) \quad \text{and}\quad\mu(E)\coloneqq\frac{\deg(E)}{\rk(E)}.\]
\begin{Prop} \label{prop:curve-to-surface-smooth}
    Let $(S, H)$ be a polarised smooth projective surface such that $K_S$ is numerically equivalent to $H$. Suppose for a smooth curve $C \in |H|$, we have 
\begin{equation} \label{eq:assumption-bn-chi}
\bn_C < \chi(\cO_S) .
\end{equation}
    Then there exists $\epsilon>0$ such that \hyperref[BGn]{\ensuremath{\mathbf{{BG}_{2}(\epsilon)}}} holds for $(S, H)$. 
\end{Prop}
\begin{proof}
Assume for a contradiction that \hyperref[BGn]{\ensuremath{\mathbf{BG_{2}(\epsilon)}}} fails for $(S, H)$ and some $0<\epsilon \ll 1$. Then there is a reflexive $H$-stable sheaf (hence locally free) $E$ as described in Lemma \ref{lem:restriction}. In particular, we know
\begin{equation}\label{assumption}
    0 < \mu_H(E) \leq \frac{2\epsilon}{1 - \epsilon} \quad\text{and}\quad 
    \ch_2(E) \geq -\frac{1}{2}\, \ch_1(E). H.
\end{equation} 
By the Hirzebruch--Riemann--Roch formula and Serre duality, we obtain 
\begin{equation*}\label{eq:chiEE}
\begin{aligned}
-\Delta(E)+\chi(\cO_S)\ch_0(E)^2=\chi(E, E) 
&\leq \mathrm{hom}(E, E) +\mathrm{hom}(E, E(K_S)) \\
&\leq  1+ 1+ \mathrm{hom}(E,  E(K_S)|_C), 
\end{aligned}
\end{equation*}
where second inequality is obtained by applying $\Hom(E, -)$ to
the exact sequence $$0 \to E(K_S-H) \to E(K_S) \to E(K_S)|_C \to 0.$$ Together with the Hodge index theorem, we then get 
\begin{equation}\label{h-1}
    -\frac{1}{H^2}\Delta_H(E)+\chi(\cO_S)\ch_0(E)^2 \leq 2 + h^0\big((E^{\vee} \otimes E(K_S))|_C\big). 
\end{equation}

By Lemma \ref{lem:restriction}, each HN factor $F_i$ of the restriction $E|_C$ satisfies
\[0 \leq \mu(F_i) = H^2\cdot \mu_{H_C}(F_i) \leq H^2\cdot\frac{2\epsilon}{1-3\epsilon}. \]
The HN filtration of $E|_C$ induces a filtration of $(E \otimes E^\vee (K_S))|_C$ whose factors are semistable bundles $F_i \otimes F_j^\vee (K_S|_C)$ by \cite{gieseker:tensor}, with slope 
\[\mu\big(F_i \otimes F_j^\vee (K_S|_C)\big) \in 
\bigg[g-1- \frac{2\epsilon}{1-3\epsilon} H^2,\, g-1+ \frac{2\epsilon}{1-3\epsilon} H^2\bigg]\] 
as $g-1 = \frac{1}{2}H.(H+K_S) = H^2$. If we choose $\epsilon$ sufficiently small, then by the definition of $\bn_C$ and \eqref{eq:assumption-bn-chi}, we have
\begin{equation*}
    \frac{h^0\big(F_i\otimes F_j^\vee(K_S|_C)\big)}{\rk\big(F_i\otimes F_j^\vee(K_S|_C)\big)}
    \leq \frac{\chi(\cO_S) +\bn_C}{2}
\end{equation*}
for all $i,j$. Summing $h^0(F_i \otimes F_j^\vee(K_S|_C))$ over all $i,j$, we obtain 
\[\frac{h^0\big(E\otimes E^\vee(K_S)|_C\big)}{\rk\big(E\otimes E^\vee(K_S)|_C\big)} 
\leq \frac{\chi(\cO_S) +\bn_C}{2},\]
which together with \eqref{h-1} implies
\begin{align*}
-\frac{1}{H^2}\Delta_H(E)+\chi(\cO_S)\ch_0(E)^2 \leq 2+ \frac{\chi(\cO_S) +\bn_C}{2}\ch_0(E)^2. 
\end{align*}
Dividing both sides by $2H^2\ch_0(E)^2$ and rearranging, we get 
\begin{align*}\label{eq-ch2-smooth}
\frac{\ch_2(E)}{H^2\ch_0(E)} & \leq 
\frac{1}{2}\mu_H(E)^2+
\frac{1}{2H^2}\left( 
\frac{\bn_C - \chi(\cO_S) }{2}+\frac{2}{\ch_0(E)^2}
\right)\\
& < -\frac{1}{2}\mu_H(E),
\end{align*}
where the second inequality holds when $\mu_H(E)$ is sufficiently small by choosing $\epsilon$ small enough in \eqref{assumption}, since $\bn_C<\chi(\cO_S)$ and $$\ch_0(E)H^2=\frac{\ch_1(E).H}{\mu_H(E)}\geq \frac{1}{\mu_H(E)}.$$ Hence 
$\ch_2(E) < -\tfrac{1}{2}\ch_1(E). H$, contradicting the assumption on $\ch_2(E)$ in \eqref{assumption}.
\end{proof}

Without the smoothness assumption, the self tensor product of a reflexive sheaf $F$ may not be torsion-free, and $\chi(F, F)$ may not be well-defined. In the following, we use an alternative approach to prove a weaker version of Proposition \ref{prop:curve-to-surface-smooth} for mildly singular $S$ and $C$.

Recall that an integral projective curve $C$ is Cohen--Macaulay, hence the property of dualizing sheaf $\omega_C$ implies
\begin{equation}\label{eq-serre-dual}
    \mathrm{ext}^i(E,\omega_C)=h^{1-i}(E)
\end{equation}
for any coherent sheaf $E$ on $C$. In particular, we get $h^i(\omega_C)=h^{1-i}(\cO_C)$ for any $i$ and $\deg(\omega_C)=2g-2$. Moreover, if $C$ is Gorenstein, then according to \cite[Lemma 1.1]{hartshorne:divisor-on-gorenstein-curve}, we have
\[\cE xt_C^i(E,\cO_C)=0\]
for any torsion-free sheaf $E$ on $C$ and $i\neq 0$. In particular, the usual dual coincides with the derived dual for torsion-free sheaves, and we have $E^{\vee\vee}=E$, $\mu(E)=-\mu(E^{\vee})$, and
\[h^0(E^{\vee})=\mathrm{hom}(E, \cO_C)=h^1(E\otimes \omega_C)\]
as in the smooth case. Therefore, the basic theory of (semi)stable sheaves over smooth curves applies without any change in this case.

\begin{Prop} \label{prop:curve-to-surface}
    Let $(S, H)$ be a polarised normal projective surface with rational Gorenstein singularities such that $K_S$ is numerically equivalent to $H$. Suppose for an integral curve $C \in |H|$, we have 
\begin{equation} \label{eq:assumption-bn}
\bn_C < \chi(\cO_S) - 1.
\end{equation}
    Then there exists $\epsilon>0$ such that \hyperref[BGn]{\ensuremath{\mathbf{BG_{2}(\epsilon)}}} holds for $(S, H)$. 
\end{Prop}
\begin{proof}
 We assume again that \hyperref[BGn]{\ensuremath{\mathbf{BG_{2}(\epsilon)}}} fails for $(S, H)$ and some $0<\epsilon \ll 1$, and consider the reflexive $H$-stable sheaf $E$ described in Lemma \ref{lem:restriction}. By the Hirzebruch--Riemann--Roch formula (cf.~Lemma \ref{lem:property-ch}(a)) and Serre duality, we have 
\begin{equation} \label{eq-RR+SD}
\ch_2(E)-\frac{\ch_1(E). H}{2}+\chi(\cO_S)\ch_0(E)
=\chi(E) \leq h^0(E)+h^0(E^\vee(K_S)).     
\end{equation}

Consider the exact sequences
\begin{align*}
&0 \to E(-H) \to E \to E|_C \to 0, \\
&0 \to E^\vee(K_S-H) \to E^\vee(K_S) \to E^\vee(K_S)|_C \to 0. 
\end{align*}
By the stability of $E$ and the assumption on its slope, we have $h^0(E(-H))=h^0(E^\vee(K_S-H))=0.$ Using the sequences above together with \eqref{eq-RR+SD}, we obtain
\begin{equation} \label{eq:restr-to-C-1}
\chi(E) \leq h^0(E|_C)+h^0(E^\vee(K_S)|_C).     
\end{equation}

By Lemma \ref{lem:restriction}, each HN factor $F_i$ of the restriction $E|_C$ satisfies 
\[0 \leq \mu(F_i) \leq H^2 \cdot\frac{2\epsilon}{1-3\epsilon}. \]
Choosing $\epsilon$ sufficiently small and applying \cite[Theorem 2.1]{newstead:geography-of-brill-noethr-loci} (see also Lemma \ref{lem:weak-bound}) implies
\[\frac{h^0(F_i)}{\rk(F_i)} \leq 1+\frac{\mu(F_i)}{2}.\]
Summing over all factors, we get
\begin{equation} \label{eq:Cliff(F)-1}
h^0(E|_C) \leq \ch_0(E)+\frac{\ch_1(E). H}{2}. 
\end{equation}

On the other hand, since $S$ is Gorenstein and integral, we know that $E$ and $E^{\vee}$ are Cohen--Macaulay and torsion-free. Hence, if we denote by $\iota\colon C\hookrightarrow S$ the embedding, we have $E^{\vee}\cong R\cH om_S(E,\cO_S)$, $E|_C\cong L\iota^*(E)$, and $E^{\vee}|_C\cong L\iota^*(E^{\vee})$. Thus, we get
\[E^{\vee}|_C\cong L\iota^*(R\cH om_S(E,\cO_S))\cong R\cH om_C(E|_C,\cO_C)\cong (E|_C)^{\vee},\]
where the second isomorphism follows from \cite[\href{https://stacks.math.columbia.edu/tag/0GM7}{Tag 0GM7(3)}]{stacks-project}. Therefore. each HN factor $G_j$ of $E^\vee(K_S)|_C\cong (E|_C)^{\vee}(K_S|_C)$ is the twist of the dual of a factor of $E|_C$, and we have
\[\mu(G_j) \in 
\left[g-1-\frac{2\epsilon}{1-3\epsilon}H^2,\ g-1\right].\]
Hence, choosing $\epsilon$ small enough and summing over all factors yields
\begin{equation} \label{eq:Cliff(Fdual)-1}
h^0(E^\vee(K_S)|_C) \leq  
\frac{\bn_C + \chi(\cO_S)-1}{2}\,\ch_0(E).
\end{equation}

Combining \eqref{eq-RR+SD}, \eqref{eq:restr-to-C-1}, \eqref{eq:Cliff(F)-1}, and \eqref{eq:Cliff(Fdual)-1}, we get
\begin{align*}
\ch_2(E)-\frac{\ch_1(E). H}{2}+\chi(\cO_S)\ch_0(E)
\leq 
\ch_0(E)+\frac{\ch_1(E). H}{2}
+\frac{\bn_C + \chi(\cO_S)-1}{2}\,\ch_0(E).
\end{align*}

Dividing both sides by $H^2\ch_0(E)$ and rearranging, we obtain
\begin{align*}
\frac{\ch_2(E)}{H^2\ch_0(E)}
\leq\;
\mu_H(E)
+\frac{1}{H^2}
\left(\frac{\bn_C - \chi(\cO_S)+1}{2}\right)
< -\frac{1}{2}\mu_H(E),
\end{align*}
where the last inequality holds for $\epsilon$ small enough as $\bn_C - \chi(\cO_S)+1<0$. This contradicts the condition on $\ch_2(E)$ in Lemma \ref{lem:restriction}.
\end{proof}

\section{Bounds on the Brill--Noether number}\label{sec:bn-bound}

In this section, we derive various bounds on the dimension of global sections of semistable sheaves on integral curves.

The following definition generalizes $\bn_C$.

\begin{Def}
    Let $C$ be an integral projective curve of genus $g\geq 1$. We define the \emph{Brill--Noether (BN) function} $\Psi_{\bn,C} \colon \R \to \R\cup\{+\infty\}$ as
    \[\Psi_{\bn, C} (x) = \limsup_{q \to x} \sup \left\{ \frac{h^0(E)}{\rk(E)} \colon \ \text{$E$ is a stable sheaf on $C$ with } \mu(E) =q    \right\}. \]
\end{Def}

The function $\Psi_{\bn, C} (x)$ is upper semicontinuous and we have
\[\frac{h^0(E)}{\rk(E)}\leq \Psi_{\bn, C}(\mu(E))\]
for any semistable sheaf $E$ on $C$. It is straightforward to check $\Psi_{\bn, C}(2g-2)=g=h^0(\omega_C)$. Moreover, by \eqref{eq-serre-dual}, we have
\[\Psi_{\bn, C} (x)=
\begin{cases}
  x+1-g, & \text{if }x\in (2g-2,+\infty) \\
  0, & \text{if }x\in (-\infty,0).
\end{cases}\]
Therefore, we only consider the restriction of $\Psi_{\bn, C} (x)$ to the interval $[0,2g-2)$. This function is previously studied in the context of higher rank Brill--Noether theory of smooth curves; see \cite{bigas:brill-noether-for-stable-vector-bundle} for a survey. For our purposes, we are only interested in its value at $g-1$, since
\[\bn_C = \Psi_{\bn, C}(g-1).\]

\begin{Rem}\label{rmk:left-limit}
When $C$ is Gorenstein, Serre duality yields
\[\Psi_{\bn, C} (g-1+t)-\Psi_{\bn, C} (g-1-t)=t\]
for any $t\in [0,g-1]$. In particular, we get
\[\bn_C = \lim_{t\to 0^+} \sup \left\{ \frac{h^0(E)}{\rk(E)} \colon \ \text{$E$ is a stable sheaf on $C$ with } \mu(E)\in (g-1-t,g-1]    \right\}. \]
We will use this equivalent definition of $\bn_C$ without mentioning.
\end{Rem}

\subsection{Classical bounds}\label{subsec:classical-bound}

First, we review some classical bounds on $\Psi_{\bn, C}(x)$ and $\bn_C$. We begin with the most general bound.

\begin{Lem}[{G. Xiao}]\label{lem:weak-bound}
Let $C$ be an integral projective curve $C$ of genus $g\geq 1$. Then 
\[\Psi_{\bn, C}(x)\leq \frac{x}{2}+1\]
for any $x\in [0,2g-2]$. In particular,
\[\bn_C\leq \frac{g-1}{2}+1.\]
\end{Lem}

\begin{proof}
Given a semistable sheaf $E$ on $C$ with $\mu(E)\in [0,2g-2]$, we aim to show that
\[h^0(E)\leq \frac{1}{2}\deg(E)+\rk(E).\]
When $\rk(E)=1$, this is \cite[Theorem A]{eisenbud:det-eq}. Then the general case follows from the same induction argument as in \cite[Theorem 2.1]{newstead:geography-of-brill-noethr-loci}.
\end{proof}

The bound in Lemma \ref{lem:weak-bound} is not optimal in general. For instance, we have the following improvement using the Clifford index. Recall that the \emph{Clifford index} of a smooth projective curve $C$ of genus $g\geq 2$ is defined as
\[\mathrm{Cliff}(C)\coloneqq \min_{L\in \Pic(C)}\left\{ \deg L-2(h^0(L)-1) \colon \  h^0(L)\geq 2, ~ \deg(L)\leq g-1\right\}\]
when $g\geq 4$. When $g=2$, we set $\mathrm{Cliff}(C)=0$, while, when $g=3$, its Clifford index is defined to be $\mathrm{Cliff}(C)=0$ or $1$ according to whether $C$ is hyperelliptic or not. It is well-known that $\mathrm{Cliff}(C)=0$ if and only if $C$ is hyperelliptic, and $\mathrm{Cliff}(C)=1$ if and only if $C$ is trigonal or a planar quintic. If we define the \emph{gonality} $\gon(C)$ of $C$ to be the lowest degree of a non-constant morphism from $C$ to the projective line, then $\mathrm{Cliff}(C)\geq 2$ if and only if $\gon(C)\geq 4$ and $C$ is not a planar quintic.

\begin{Lem}[{\cite{re:multiplication,mercat:clifford-theorem}}]\label{lem:mu/2-bound}
Let $C$ be a smooth projective curve of genus $g\geq 2$. Then we have
    \[\bn_C\leq \frac{g-1}{2}+1-\frac{\min\{\mathrm{Cliff}(C),2\}}{2}.\]
\end{Lem}

\begin{proof}
When $\mathrm{Cliff}(C)=0$, the assertion follows from Lemma \ref{lem:weak-bound}. When $\mathrm{Cliff}(C)=1$, it follows from \cite[Proposition 3 and 4]{re:multiplication}. Thus, we may assume that $g\geq 5$ and $\mathrm{Cliff}(C)\geq 2$, in which case \cite[Theorem 2.1]{mercat:clifford-theorem} gives the desired inequality.
\end{proof}

In practice, the Castelnuovo--Severi inequality is useful for computing the gonality.

\begin{Lem}[{\cite[Theorem 3.5]{accola:topics-in-rim-surface}}]\label{lem:cast-severi}
Let $f\colon C\to C_1$ and $g\colon C\to C_2$ be non-constant morphisms between smooth projective curves. Assume that there does not exist a non-trivial morphism $h\colon C\to C'$ such that $f$ and $g$ both factor through $h$, then
\[g(C)\leq g(C_1)\deg f+g(C_2)\deg g+(\deg f-1)(\deg g-1).\]
\end{Lem}

\subsection{Bounds on $\bn_C$ via wall-crossing}\label{subsec:bn-wall-cross}

Next, we deal with the case when $C$ can be embedded into a del Pezzo or K3 surface. This method is based on the wall-crossing argument used in \cite{bayer:brill-noether,feyz:mukai-program,feyz-li:clifford-indices}.

Let $S$ be a normal projective surface with rational Gorenstein singularities with an ample divisor $H$. Then the notion of Chern characters and tilt-stability on $S$ behaves well as discussed in Section \ref{sec:pre}. We fix an integer $s \geq 1$ and an integral curve $C \in |sH|$ of arithmetic genus $g$, which by the adjunction formula satisfies
\[2(g-1) = s^2 H^2 + s H. K_S.\]
Write the associated embedding by $\iota\colon C\hookrightarrow S$. Fix a coherent sheaf $E$ on $C$ of rank $r$ and degree $d$. Then Lemma \ref{lem:property-ch}(a) implies
\[\ch(\iota_*E)=\left(0,rsH,d-r\frac{s^2H^2}{2}\right)\]
and so 
$$\nu_{BN}(\iota_*E) = \frac{\mu(E)}{sH^2} - \frac{s}{2}. $$

The following result will be used frequently.

\begin{Prop}\label{prop:bn-bound}
Using the above notation, for any stable sheaf $E$ on $C$ with $\mu(E)\in [0,s^2H^2]$, there exists $\delta > 0$ such that the HN filtration of 
$\iota_*E$ with respect to $\nu_{0,w}$ is a fixed sequence
\[0 = F_0 \subset F_1 \subset \cdots \subset F_k = \iota_*E\]
for all $0 < w < \delta$. In particular, each factor 
$F_{j+1}/F_j \in \Coh^0(S)$ is BN-semistable, and we have
\[\nu_{BN}^+(\iota_*E) \coloneqq \nu_{BN}(F_1) 
    \leq \frac{2\,\nu_{BN}(\iota_*E) + s}{4},
    \quad
\nu_{BN}^-(\iota_*E) \coloneqq \nu_{BN}(\iota_*E/F_{k-1})
    \geq \frac{2\,\nu_{BN}(\iota_*E) - s}{4}.\]
\end{Prop}

\begin{proof}
By the stability of $E$, we know $\iota_*E\in \Coh^0(S)$ is $\nu_{0,w}$-stable for any $w\gg 0$. If $\iota_*E$ is BN-semistable, then $\nu_{BN}^+(\iota_*E)=\nu_{BN}^-(\iota_*E)=\nu_{BN}(\iota_*E)$ and the result follows from $\mu(E)\in [0,s^2H^2]$. Thus, we may assume that $\iota_*E$ is not BN-semistable. It is proved in \cite[Proposition 3.4(a)]{feyz:mukai-program} that the HN filtration of $\iota_*E$ with respect to $\nu_{0,w}$ is independent of $w$ for $0<w<\delta$. Hence, it remains to prove the bounds on $\nu_{BN}^{\pm}(\iota_*E)$. 

Let $E_2 \to \iota_*E \to E_1$ be the destabilizing sequence along the first wall $\ell$ that we hit for $\iota_*E$ when we move down from $w \gg 0$ along the vertical line $b=0$. Suppose $\ell$ intersects $w=\frac{1}{2}b^2$ at $(b_1,\frac{1}{2}b_1^2)$ and $(b_2,\frac{1}{2}b_2^2)$, where $b_1<0<b_2$. We show that
\begin{equation}\label{eq-1}
b_2-b_1 \leq s,    
\end{equation}
and so the line $\ell$ lies on or below the line connecting the points
\[\left(\nu_{\bn}(\iota_*E) - \frac{s}{2}, \ \frac{1}{2}\bigl(\nu_{\bn}(\iota_*E) - \frac{s}{2}\bigr)^2\right)\]
and
\[\left(\nu_{\bn}(\iota_*E) + \frac{s}{2}, \ \frac{1}{2}\bigl(\nu_{\bn}(\iota_*E) + \frac{s}{2}\bigr)^2\right)\]
on the parabola $w=\frac{1}{2}b^2$, and then the bounds on $\nu_{BN}^{\pm}(\iota_*E)$ follow from \cite[Lemma 3.4]{feyz-li:clifford-indices} by replacing the curve $\Gamma$ there with the parabola $w=\frac{1}{2}b^2$.

To prove \eqref{eq-1}, we take the cohomology of the destabilizing sequence and obtain the long exact sequence
\[0\to \cH^{-1}(E_1)\to E_2\to \iota_*E\to \cH^0(E_1)\to 0.\]
Therefore, we can assume that $\rk(E_2)=\rk(\cH^{-1}(E_1))=l$ for $l\geq 0$. Note that $l>0$, otherwise $\cH^{-1}(E_1)=0$ as it is torsion-free, and $E_2$ is a subsheaf of $\iota_*E$ which cannot make a wall for $\iota_*E$ by the stability of $E$.

Since $\cH^{-1}(E_1)$ is torsion-free, we have $T(E_2)\subset \iota_*E$, where $T(E_2)$ is the torsion part of $E_2$. Hence, we may assume that $\ch_1(T(E_2))=asH$ for some $a\in \Z_{\geq 0}$. We also know $\cH^0(E_1)$ is supported on $C$, so $\ch_1(\cH^0(E_1))=csH$ for some $c\in \Z_{\geq 0}$.
Hence 
\begin{align*}
    r-c=\rk(E)-\rk(\iota^*\cH^0(E_1))\leq & \rk(\iota^*E_2)\\
    \leq & \rk(\iota^*(E_2/T(E_2)))+\rk(\iota^*T(E_2))=l+a.
\end{align*}
As $$\ch_1(E_2)-\ch_1(\cH^{-1}(E_1))=\ch_1(\iota_*E)-\ch_1(\cH^0(E_1))=(r-c)sH,$$
we get
\begin{equation*}
    b_2-b_1 \leq \mu_H(E_2/T(E_2))-\mu_H(\cH^{-1}(E_1))=\frac{(r-c)sH^2-asH^2}{lH^2}\leq s,
\end{equation*}
where the first inequality comes from Lemma \ref{locally finite set of walls} and the definition of $\Coh^b(S)$ as we move along the line $\ell$. This completes the proof of \eqref{eq-1}.
\end{proof}

Once we have good control of $h^0$ of BN-semistable objects on $S$, we can bound $\bn_C$ effectively using the above result. In the following, we will realize this strategy for del Pezzo and K3 surfaces.

\subsubsection{Del Pezzo surfaces}

We first assume that $S$ is a del Pezzo surface with rational Gorenstein singularities and $H=-K_S$. Using Kodaira vanishing, we have $\chi(\cO_S)=1$, which together with Lemma \ref{lem:property-ch}(a) gives 
\begin{equation}\label{eq-dp-hrr}
    \chi(F)=\rk(F)+\frac{\ch_1(F).H}{2}+\ch_2(F)
\end{equation}
for any $F\in \mathrm{D^b}(S)$. By the adjunction formula and setting $m\coloneqq H^2$, we get
\[2g-2=s(s-1)m.\]

We define a function
\[\Psi(x,y)\coloneqq\begin{cases}
     \frac{1}{2}y+x, & \text{if }-\frac{1}{2}<\frac{x}{y}<+\infty \\
      \frac{1}{mn}y, & \text{if }\frac{x}{y}=-\frac{n}{2} \text{ for }n\in \mathbb{Z}_{>0} \\
      \frac{(2n+1)}{(n^2+n)m+2}y+\frac{2}{(n^2+n)m+2}x, & \text{if }\frac{x}{y}\in (-\frac{n+1}{2},-\frac{n}{2})\text{ for }n\in \mathbb{Z}_{>0}.
      \end{cases}\]

\begin{Rem}\label{rem:Psi}
A simple computation confirms that for $y_0>0$ and $n \in \mathbb{Z}_{>0}$, we have 
\[\max\left\{\lim_{\substack{ (x, y) \to \left(-\frac{n}{2}y_0,\, y_0\right) \\ \frac{x}{y} > -\frac{n}{2} }}\Psi(x,y) \ , \ \lim_{\substack{ (x, y) \to \left(-\frac{n}{2}y_0,\, y_0\right) \\ \frac{x}{y} < -\frac{n}{2} }}\Psi(x,y) \right\}
\le
\Psi\!\left(-\frac{n}{2}y_0,\, y_0\right).\]
\end{Rem}

Using \eqref{eq-dp-hrr} and the same argument as in \cite[Lemma 4.8]{chunyi:stability-condition-quintic-threefold}, we have:

\begin{Lem}\label{lem:h^0-bn-dp}
Let $F\in \Coh^0(S)$ be a BN-semistable object with $\ch_1(F).H>0$. Then we have
\[h^0(F)\leq \rk(F)+\Psi\big(\ch_2(F),\ch_1(F).H\big).\]
\end{Lem}

Similarly, we get:

\begin{Lem}[{\cite[Lemma 4.11]{chunyi:stability-condition-quintic-threefold}}]\label{lem:convex-dp}
Let $O=(0,0)$ be the origin, and let $P=(x_p,y_p)$ and $Q=(x_q,y_q)$ be two points in the upper half-plane such that
\[\frac{x_p}{y_p} < \frac{x_q}{y_q}
\quad\text{and}\quad
y_p > y_q.\]
Consider the collection of all convex polygons given by a sequence of points
\[P_0=O,\, P_1,\, \dots,\, P_n=P\]
in the upper half-plane and inside the triangle $OPQ$. The maximal value of
\[\sum_{i=0}^{n-1} \Psi\!\left(\overrightarrow{P_i P_{i+1}}\right)\]
is attained by a polygon with $n=1$ or $n=2$. Moreover, when $n=2$, 
the point $P_1=(x_1,y_1)$ can be chosen to lie on the line segment $OQ$ (resp.~$QP$) unless
\[\frac{x_1}{y_1}=-\frac{k}{2}
\quad\text{(resp.~}\frac{x_p-x_1}{y_p-y_1}=-\frac{k}{2}\text{)}\]
for some $k\in\mathbb{Z}_{>0}$.
\end{Lem}

Now, we can prove a bound on $\bn_C$.

\begin{Prop}\label{prop:delpezzo}
When $s$ is odd, we have
\[\bn_C\leq \max\left\{1+\frac{1}{8}H^2(s^2-1),\ s\right\}.\]
\end{Prop}

\begin{proof}
We may assume that $s\geq 3$, as $s=1$ implies $g(C)=1$ and the result follows from Lemma \ref{lem:weak-bound}. Let $E$ be a stable sheaf on $C$ of rank $r$ and degree $d$ with slope 
\begin{equation}\label{mu}
\mu(E)\in \left[\frac{1}{2}s(s-1)m-\epsilon,\ \frac{1}{2}s(s-1)m\right]    
\end{equation}
for $0<\epsilon\ll 1$.
If $\iota_*E$ is BN-semistable, then Lemma \ref{lem:h^0-bn-dp} directly implies $h^0(E)\leq rs$. Now, assume that $\iota_*E$ is not BN-semistable. Let $$0=F_0\subset F_1\subset F_2\subset \cdots \subset F_k=\iota_*E$$ be the filtration of $\iota_*E$ as in Proposition \ref{prop:bn-bound}. Let $$P_i\coloneqq (\ch_2(F_i), \ch_1(F_i).H)$$ be the corresponding points in the upper half plane. By Lemma \ref{lem:h^0-bn-dp}, we have
\[h^0(E)\leq \sum^k_{i=1}h^0(F_i/F_{i-1})\leq \sum^k_{i=1} \Psi(\overrightarrow{P_{i-1} P_{i}}).\]
We set $$P=(x_p,y_p)\coloneqq P_k=\left(d-\frac{1}{2}s^2rm,\, rsm\right)$$ and $$Q\coloneqq(x_q,y_q)=\left(\frac{\mu(E)d}{2s^2m},\, \frac{d}{s}\right).$$
So we have
$$\frac{x_q}{y_q}=\frac{\mu(E)}{2sm} \quad \text{and} \quad 
\frac{x_q-x_p}{y_q-y_p}=\frac{\mu(E)-s^2m}{2sm}.$$ 
By Proposition \ref{prop:bn-bound}, the points $O,P_1,\dots,P_k,O$ form the vertices of a convex polygon in the triangle $OQP$. Using Lemma \ref{lem:convex-dp}, to bound $h^0(E)$, we only need to bound
\[\Psi(\overrightarrow{O P_{1}})+\Psi(\overrightarrow{P_{1} P})\]
by choosing suitable point $P_1=(x_1,y_1)$ in the triangle $OQP$. Note that when $\frac{x_1}{y_1}=-\frac{n}{2}$ for $n\in \Z_{>0}$, by Proposition \ref{prop:bn-bound}, we have $n=1$. Similarly, when $\frac{x_p-x_1}{y_p-y_1}=-\frac{n}{2}$ for $n\in \Z_{>0}$, by Proposition \ref{prop:bn-bound}, we have $1\leq n\leq \frac{s+1}{2}$. So the argument can be divided into the following cases.

\begin{itemize}
    \item When $\frac{x_p-x_1}{y_p-y_1}=-\frac{1}{2}$, then we have $\mu(E)=\frac{1}{2}s(s-1)m$ and a direct computation gives
\[\frac{h^0(E)}{r}\leq \frac{1}{r}\big(\Psi(\overrightarrow{O P_1})+\Psi(\overrightarrow{P_1 P})\big)=\frac{1}{r}\left(\frac{1}{m}y_1+\frac{1}{m}(y_p-y_1)\right)=s.\]

    \item Suppose $P_1=Q$. Then, by Remark \ref{rem:Psi} and the slope range in \eqref{mu}, we have
    \begin{align*}
        \frac{h^0(E)}{r} 
\le 
\frac{1}{r} \bigl( \Psi(\overrightarrow{OQ}) + \Psi(\overrightarrow{QP}) \bigr) 
\le \,& 
\frac{1}{r} \Biggl( \frac{1}{2} y_q + x_q 
+ \frac{2}{m(s+1)} (y_p - y_q) \Biggr) 
\Biggr|_{\mu(E) = \frac{1}{2}s(s-1)m}\\
= \,& 1+\frac{1}{8}m(s^2-1). 
    \end{align*}

\item When $\frac{x_1}{y_1}=-\frac{1}{2}$ and $\frac{x_p-x_1}{y_p-y_1}=-\frac{n}{2}$ for $2\leq n\leq \frac{s+1}{2}$, we have
   $$y_1 =  -r\frac{sm(s-n)-2\mu(E)}{n-1}.$$
Then we get
    \[\frac{h^0(E)}{r}\leq \frac{1}{r}\big(\Psi(\overrightarrow{O P_1})+\Psi(\overrightarrow{P_1 P})\big)=\frac{1}{r}\left(\frac{1}{m}y_1+\frac{1}{mn}(y_p-y_1)\right)\]
For any fixed $n$, this bound is a linear function in $\mu(E)$. Moreover, its value at $\mu(E)=\frac{1}{2}s(s-1)m$ always equals $s$.

\item When $\frac{x_1}{y_1} = -\frac{1}{2}$ and $P_1$ is on $PQ$, then, again by Remark \ref{rem:Psi} and the slope range in \eqref{mu}, we have
\[\frac{h^0(E)}{r} \le \frac{1}{r} \bigl(\Psi(\overrightarrow{O P_1}) +\Psi(\overrightarrow{P_1 P}) \bigr) \le \frac{1}{r}\Biggl( \frac{1}{m} y_1 + \frac{2}{m(s+1)} (y_p -y_1) \Biggr)\Biggr|_{\mu(E) = \frac{1}{2}s(s-1)m} =s\]
where the last equality follows from the previous part.

\item When $\frac{x_p-x_1}{y_p-y_1}=-\frac{n}{2}$ for $2\leq n\leq \frac{s+1}{2}$ and $P_1$ is on $OQ$, we have
$$y_1 = \frac{2\mu(E)+sm(n-s)}{\mu(E)+smn}smr. $$
Then we obtain
\[\frac{h^0(E)}{r}\le\frac{1}{r}\bigl(\Psi(\overrightarrow{O P_1})+\Psi(\overrightarrow{P_1 P})\bigr)=\frac{1}{r}\left(\frac{1}{2}y_1+\frac{\mu(E)}{2sm}y_1+\frac{1}{mn}(y_p-y_1)\right),\]
and this bound is clearly a rational function of $\mu(E)$. Its value at $\mu(E)=\frac{s(s-1)m}{2}$ equals
$$
 \frac{2sm(n-1)}{s+2n-1} \left(\frac{s+1}{4} -\frac{1}{mn}\right)+ \frac{s}{n}. 
$$
This is an increasing function of $n$ for $n\in [2, \frac{s+1}{2}]$ when $s\ge 3$, and hence its maximum equals
$1+\frac{1}{8}m(s^2-1)$. 
\end{itemize}
\end{proof}

The case where $s$ is even can be dealt with analogously. Since we only need the odd case in the later section, we skip this part.

\subsubsection{K3 surfaces}

Now, we discuss the case where $C$ lies in a K3 surface. In the rest of this section, we fix $S$ to be a smooth\footnote{All results can be generalized to K3 surfaces with mild singularities. However, this requires additional effort, and we do not address it in the current paper.} K3 surface and $H$ to be an ample divisor on $S$ such that $H^2$ divides $D.H$ for all divisors $D$ on $S$ (e.g.~$\mathrm{Pic}(S)=\mathbb{Z}H$). Then by setting $m\coloneqq H^2$, we get
\[s^2m=2g-2.\]
For any coherent sheaf $E$ of rank $r$ and degree $d$ on $C$, we have
\[\ch(\iota_*E)=\big(0,rsH, d+r(1-g)\big)=\left(0,\,rsH,\,r\left(\mu(E)-\frac{s^2}{2}m\right)\right).\]
In particular, if $\mu(E)=g-1$, then $\ch(\iota_*E)=(0,rsH, 0).$ We define a function $$\Omega(x,y)\coloneqq\frac{x}{2}+\frac{1}{2}\sqrt{x^2+\frac{2m+4}{m^2}y^2}.$$
Now, we obtain an upper bound for $\bn_C$ via an argument similar to that in Proposition \ref{prop:delpezzo}.

\begin{Prop}\label{prop:K3}
We have
\[\bn_C\leq \frac{s}{8}\sqrt{(2m+8)^2+(s^2-4)m^2}.\]
\end{Prop}

\begin{proof}
Let $E$ be a stable sheaf on $C$ of rank $r$ and degree $d$ with 
\[\mu(E)\in \left[\frac{1}{2}s^2m-\epsilon,\ \frac{1}{2}s^2m\right] \]
for $0<\epsilon\ll 1$. Let $$0=F_0\subset F_1\subset F_2\subset \cdots \subset F_k=\iota_*E$$ be the filtration of $\iota_*E$ as in Proposition \ref{prop:bn-bound}. We set $$P_i=(x_i,y_i)\coloneqq\left(\ch_2(F_i), \ch_1(F_i).H\right)$$ to be the corresponding points in the upper half plane. From \cite[Proposition 3.4(b)]{feyz:mukai-program}, we have
\[h^0(E)\leq \sum^k_{i=1}h^0(F_i/F_{i-1})\leq \sum^k_{i=1} \Omega(\overrightarrow{P_{i-1} P_{i}}).\]
We set $$P=(x_p,y_p)\coloneqq P_k=\left(d-\frac{s^2}{2}rm, rsm\right)$$ and $$Q\coloneqq(x_q,y_q)=\left(\frac{\mu(E)d}{2s^2m},\frac{d}{s}\right)$$
so that $$\frac{x_q}{y_q}=\frac{\mu(E)}{2sm} \quad \text{and} \quad \frac{x_q-x_p}{y_q-y_p}=\frac{\mu(E)-s^2m}{2sm}.$$  Therefore, by Proposition \ref{prop:bn-bound}, the points $O,P_1,\dots,P_k,O$ form the vertices of a convex polygon in the triangle $OQP$. By the triangle inequality, we have
\begin{align*}
\frac{h^0(E)}{r}
&\leq \frac{1}{r}\sum_{i=1}^k \Omega\!\left(\overrightarrow{P_{i-1}P_i}\right) \\
&\leq \frac{\ch_2(\iota_*E)}{2r}
+ \sum_{i=1}^k \frac{1}{2r}
\sqrt{(x_i-x_{i-1})^2 + \frac{2m+4}{m^2}(y_i-y_{i-1})^2} \\
&\leq \frac{\mu(E)}{2} - \frac{s^2m}{4}
+ \frac{1}{2}\sqrt{\left(\frac{x_q}{r}\right)^2 + \frac{2m+4}{m^2}\left(\frac{y_q}{r}\right)^2}\\
& \ \quad \quad
+ \frac{1}{2}\sqrt{\left(\frac{x_p-x_q}{r}\right)^2 + \frac{2m+4}{m^2}\left(\frac{y_p-y_q}{r}\right)^2}.
\end{align*}
In particular, we get an upper bound of $\frac{1}{r}h^0(E)$ by a continuous function in $\mu(E)$. Now the result for $\bn_C$ follows from a direct computation at $\mu(E)=\frac{1}{2}s^2m=g-1$.
\end{proof}

\subsection{Lower bounds on $\bn_C$}
    \label{subsec:bn-lower-bound}
    We conclude this section by describing some lower bounds on $\bn_C$. Although the results of this subsection are independent of the rest of the article, they indicate that $\bn_C$ captures information fundamentally different from existing invariants (such as the gonality and Clifford index), and is not easy to compute in general.

  Our first observation is that the gonality of a curve gives a naive lower bound on $\bn_C$.
  \begin{Lem}
    \label{lemma:gonalityLowerBound}
    If $C$ is an integral projective curve of genus $g$ with a finite morphism of degree $d$ to $\PP^1$, then
    \[\bn_C\geq \left\lfloor\frac{g-1}{d}\right\rfloor+1.\]
  \end{Lem}
  \begin{proof}
    Let $\pi\colon C\to\PP^{1}$ denote the associated degree $d$ cover.
    Consider the line bundle
    \[L=\cO_{\PP^{1}}\left(\left\lfloor\frac{g-1}{d}\right\rfloor\right),\]
    which satisfies
    \[h^{0}(\pi^*L)\geq h^0(L)=\left\lfloor\frac{g-1}{d}\right\rfloor+1.\]
    Twisting $\pi^*L$ by $g-1-d\lfloor (g-1)/d\rfloor$ general points gives a line bundle, $L^{\prime}$, with $\deg(L^{\prime})=g-1$ and $h^{0}(L')\geq h^{0}(\pi^*L)$, as claimed.
  \end{proof}
  
  In the case of smooth hyperelliptic curves, this bound is best possible.

  \begin{Lem}
    \label{lemma:bnHyperelliptic}
    If $C$ is a smooth hyperelliptic curve of genus $g\geq 1$, then
    \[\bn_C=\left\lfloor\frac{g-1}{2}\right\rfloor+1.\]
  \end{Lem}
  \begin{proof}
    If $g$ is odd, then $\bn_C\leq \frac{1}{2}(g+1)$ by Lemma \ref{lem:weak-bound}.
    If $g$ is even, then we have $\bn_C\leq \frac{1}{2}g$~\cite[Theorem 6.2]{brambila:nonempty}. Now the result follows from Lemma \ref{lemma:gonalityLowerBound} and $\gon(C)=2$.
  \end{proof}

  We could also consider an analogous argument to Lemma \ref{lemma:gonalityLowerBound} applied to minimal degree nondegenerate morphisms $C\to\PP^{n}$ with $n\geq 2$. For example, if $C$ is a planar curve, then we have:

  \begin{Lem}
    \label{lemma:bnSmoothQuintic}
    If $C$ is an integral planar curve of odd degree $d$, then 
    \[\bn_C=\frac{d^2-1}{8}.\]
  \end{Lem}
  \begin{proof}
    The line bundle 
    \[L\coloneqq \cO_{\PP^2}\left(\frac{d-3}{2}\right)|_C\]
    satisfies $\operatorname{deg}(L)=\frac{1}{2}d(d-3)=g(C)-1$ and $h^{0}(L)=\frac{1}{8}(d^2-1)$. Now the result follows from \cite[Theorem 5.5]{feyz-li:clifford-indices}, which is stated for smooth curves, but the proof works without any change in this case.
  \end{proof}

  So far, we have determined $\bn_C$ for all curves with $\mathrm{Cliff}(C)=0$ and some curves with $\mathrm{Cliff}(C)=1$.
  It would be interesting to determine $\bn_C$ for the remaining curves of Clifford index $\leq 1$ (i.e.~trigonal curves), but this does not seem easy in general:
  If $g=3$, then $\bn_C=3/2$~\cite{mercat:slope-smaller-2}.
  If $g=4,5$, then the bound $\bn_C\geq 2$ of Lemma \ref{lemma:gonalityLowerBound} is optimal \cite{lange:genus-4,lange:genus-5}.
  If $g=6$, we can improve Lemma \ref{lemma:gonalityLowerBound} using higher rank Brill--Noether theory of generic curves as follows.

  If $D$ is a generic curve of genus $6$ (in particular, $D$ is \emph{not} trigonal), then $D$ admits a stable bundle $E$ with invariants $\rk(E)=2$, $\deg(E)=10$, and $h^{0}(E)=5$~\cite{bertram:stable-rank-two-vector-bundles}.
  Then, by semicontinuity, any smooth projective curve $C$ of genus $6$ admits a semistable bundle $E^{\prime}$ with invariants $\rk(E^{\prime})=2$, $\deg(E^{\prime})=10$, and $h^{0}(E^{\prime})\geq 5$. Therefore, we have $\bn_C\geq 5/2$.
  On the other hand, \cite[Proposition 4.13]{lange:genus-6} shows $\bn_C\leq 161/60$ when $C$ is trigonal, which is strictly stronger than Lemma \ref{lem:mu/2-bound}.

  Note that the above argument only uses the fact that a lower bound on $\bn_D$ for generic curves $D$ implies the same lower bound on $\bn_C$ for any smooth projective curve of genus $g(D)$, and does not rely on the geometry of trigonal curves. More generally, applying the semicontinuity theorem countably many times, we have:

\begin{Lem}
Let $p\colon \mathcal{C}\to B$ be a flat projective family of integral curves of genus $g$ over an irreducible scheme $B$. Set $\cC_b\coloneqq p^{-1}(b)$ for $b\in B$.

\begin{enumerate}
    \item If there exists a constant $A_1>0$ such that $\bn_{\cC_b}\geq A_1$ for very general points $b\in B$, then $\bn_{\cC_b}\geq A_1$ for any $b\in B$.

    \item If there exists a constant $A_2>0$ and a point $b_0\in B$ such that $\bn_{\cC_{b_0}}\leq A_2$, then $\bn_{\cC_{b}}\leq A_2$ for any very general point $b\in B$.
\end{enumerate}

\end{Lem}

In particular, part (b) together with \cite[Theorem 1.1]{feyz-li:clifford-indices} implies that
\[\bn_C \leq \frac{g}{4}+1+\frac{1}{g}\]
for any very general curve $C$ of genus $g\geq 2$. It would be interesting to determine $\bn_C$ for such curves.

  We can also precisely determine $\bn_C$ in the case of bielliptic curves.
  
  \begin{Lem}
    \label{lemma:brillNoetherNumberbielliptic}
    Suppose $C$ is a smooth bielliptic curve of genus $g\geq 4$. If $C$ is neither hyperelliptic nor trigonal, then
    \[\bn_C= \frac{g-1}{2}.\]
  \end{Lem}
  \begin{proof}
    By \cite[Theorem 5.3]{ballico:bNCurves}, for every integer $d\geq 1$ and $r\geq 2$, we can find a stable bundle $E$ on $C$ satisfying $\rk(E)=r$, $\deg(E)=d$, and
    \[\frac{h^{0}(E)}{r}\geq \frac{1}{r}\left\lfloor\frac{d-1}{2}\right\rfloor.\]
    In particular, we get
    \[\bn_C\geq \frac{1}{r}\left\lfloor\frac{r(g-1)-1}{2}\right\rfloor \geq \frac{r(g-1)-2}{2r}.\]
    Taking the limit $r\to \infty$ shows $\bn_C\geq \frac{1}{2}(g-1)$.
   When $C$ is neither hyperelliptic nor trigonal, since there is no bielliptic planar quintic curve by \cite[Proposition 2.2(v)]{4-gonal}, we obtain $\bn_C\leq \frac{1}{2}(g-1)$ from Lemma \ref{lem:mu/2-bound}, as claimed.
  \end{proof}
  
  Using \cite[Theorem 5.6]{ballico:bNCurves}, a similar result can be proved for double covers of hyperelliptic curves that are neither hyperelliptic nor trigonal.

\section{A collection of CY threefolds }\label{sec:app}
Finally, we combine the results of the previous sections to prove \ref{conj:bmt} for a collection of Calabi--Yau threefolds. 

An immediate corollary of Theorem \ref{thm:main-criterion} is the following criterion:

\begin{Cor}\label{cor:trivial-cor-basepoint-free}
Let $(X, H)$ be a polarised Calabi--Yau threefold. Suppose that there exists a smooth surface $S\in |H|$ and a smooth curve $C\in |H_S|$ (e.g. $H$ is basepoint-free). If
\[c_2(X).H> 4H^3+12,\]
then \emph{\ref{conj:bmt}} holds for $(X,H)$ and some $1$-cycle $\Gamma$ with $\Gamma.H\geq 0$. Furthermore, if $2H$ is very ample, then it suffices to assume
\[c_2(X).H> 4H^3+6.\]
\end{Cor}

\begin{proof}
From Hirzebruch--Riemann--Roch, the first assumption is equivalent to
\begin{equation}\label{chi}
  \frac{H^3}{2}+1< \chi(\cO_X(H)).  
\end{equation}
Since $g(C)-1=H^3$, \ref{conj:bmt} for $(X,H)$ follows from Lemma \ref{lem:weak-bound} and Theorem \ref{thm:main-criterion}(a).

When $2H$ is very ample, we see that $\omega_C=\cO_C(2H_C)$ is very ample as well. In particular, $C$ is not hyperelliptic. Then the second statement follows from Lemma \ref{lem:mu/2-bound} and Theorem \ref{thm:main-criterion}(a).
\end{proof}

There exist Calabi--Yau threefolds that satisfy the assumption in Corollary \ref{cor:trivial-cor-basepoint-free}, such as examples in \cite{chunyi:stability-condition-quintic-threefold,koseki:double-triple-solids,liu:bg-ineqaulity-quadratic}. On the other hand, this assumption is restrictive since $c_2(X).H\leq 4H^3+36$ when $H$ is basepoint-free and $c_2(X).H\leq 2H^3+40$ when $H$ is very ample (cf.~\cite[Proposition 2.2]{wilson:trilinear-cy3}). In the rest of this section, we aim to get better criteria in various situations.

\subsection{Anticanonical divisors in Fano fourfolds of index $\geq 2$}

A natural source of Calabi--Yau threefolds is anticanonical divisors of smooth Fano fourfolds. For a Fano manifold $M$, the \emph{index} $r$ is defined to be the largest integer such that $$H_M\coloneqq-\frac{1}{r}K_M\in \Pic(M),$$ which satisfies $1\leq r\leq \dim M+1$. If $\dim M=4$, by the adjunction formula, any smooth divisor $X$ in $|-K_M|$ is a Calabi--Yau threefold. Using Proposition \ref{prop:delpezzo} and \ref{prop:K3}, we can prove \ref{conj:bmt} for $X$ under some additional assumptions on $M$.

\begin{Thm}\label{thm:cicy-in-fano}
Let $M$ be a smooth Fano fourfold of index $r$, $X\in |-K_M|$ be a general divisor, and $H\coloneqq H_X$. If either $r\geq 3$, or $r=2$ and $M$ has Picard number one, then \emph{\ref{conj:bmt}} holds for $(X,H)$ and some $1$-cycle $\Gamma$ with $\Gamma.H\geq 0$.
\end{Thm}

\begin{proof}
We first consider 
\begin{itemize}
    \item a general divisor $X \in |rH_M|$, 
    \item a general divisor $S \in |H|$, and 
    \item a general divisor $C \in |H_S|$. 
\end{itemize}

Alternatively, the surface $S$ can be constructed via a different procedure:

\begin{itemize}
    \item take a general divisor $X' \in |H_M|$, and then
    \item $S$ is obtained as a general divisor in $|rH_{X'}|$.
\end{itemize}
By \cite[Theorem 2.3.2]{shafarevich:fano-varieties} and \cite[Theorem 1]{mella:good-divisor}, $X'$ is a smooth Fano threefold of index $r-1$. Moreover, $rH_{X'}$ is basepoint-free by \cite[Corollary 2 and 2*]{ein-lazarsfeld:global-generation}, so $S$ is a smooth surface.

Similarly, the curve $C$ can be constructed via:
\begin{itemize} 
    \item take a general divisor $S' \in |H_{X'}|$, and then 
    \item $C$ is obtained as a general divisor in $|rH_{S'}|$. 
\end{itemize}
By \cite[Theorem 2.3.1]{shafarevich:fano-varieties}, $S'$ is a smooth del Pezzo surface for $r=3$ or a smooth K3 surface for $r=2$. Then the basepoint freeness of $|rH_{S'}|$ implies that $C$ is a smooth curve.

Combining the above argument with Theorem \ref{thm:main-criterion}(a), to prove \ref{conj:bmt}, it suffices to verify \eqref{eq:thm-bn-smooth} for $C$ constructed above. When $r=5$ or $4$, this can be checked using Lemma \ref{lem:weak-bound}, or equivalently, Corollary \ref{cor:trivial-cor-basepoint-free}. In the following, we may assume that $r=3$ or $2$.

By the above argument, $S'$ is a del Pezzo surface when $r=3$ and a K3 surface when $r=2$, and $C$ embeds into $S'$ as a divisor in $|rH_{S'}|$. In the former case, Proposition \ref{prop:delpezzo} implies $\bn_C\leq \max\{1+H_M^4,3\}$ and the result can be deduced from \cite[Corollary 2.1.14]{shafarevich:fano-varieties} since $\chi(\cO_X(H))=\chi(\cO_M(H_M))=H_M^4+3$. For $r=2$, if $M$ has Picard number one and $H_M^4\geq 6$, by taking $S'$ to be very general, we know that $S'$ is a smooth K3 surface with $\Pic(S')=\Z H_{S'}$ by \cite[Lemma 6.1(c)]{bayer:mukai-model-fano-var}. So we can apply Proposition \ref{prop:K3} to $C\hookrightarrow S'$ and obtain \[\bn_C\leq \frac{1}{2}H_M^4+2.\]
Then the result can be deduced from $$\chi(\cO_X(H))=\chi(\cO_M(H_M))=\frac{1}{2}H_M^4+4.$$
Finally, when $H_M^4\leq 4$, we have $H_X^3\leq 8$, and the result follows from Corollary \ref{cor:trivial-cor-basepoint-free}.
\end{proof}

\begin{Rem}\label{rem-equality}
    Adopting the notation from Theorem \ref{thm:cicy-in-fano}  and its proof, we conclude that
    $$\bn_C= h^0(\cO_C(H_C)) = \chi(\cO_X(H))-2$$ when $r\geq 2$ unless $r=3$ and $H^4_M=1$.
\end{Rem}

To generalize Theorem \ref{thm:cicy-in-fano} to index~$2$ Fano fourfolds with higher Picard number, which are listed in \cite[Section~12.7]{shafarevich:fano-varieties}, it suffices to find an appropriate analogue of \cite[Proposition~3.4(b)]{feyz:mukai-program} for K3 surfaces contained in such fourfolds. For instance, when $M$ is a double cover of $\PP^2 \times \PP^2$, or a complete intersection of two $(1,1)$-divisors in $\PP^3 \times \PP^3$, a naive generalization of \cite[Proposition~3.4(b)]{feyz:mukai-program} holds. This yields Conjecture~\ref{conj:bmt} for $X \in |-K_M|$. We leave a more detailed treatment to future work. Note that for some of these Fano fourfolds, it is straightforward to verify Conjecture~\ref{conj:bmt} directly, as explained below.

\begin{Ex}\label{ex:high-pic-rk}
If $M = \PP^1\times Y_1$ for a del Pezzo threefold $Y_1$ of degree $1$ and $X\in |-K_M|$ is a general divisor, then $\chi(\cO_M(H_M))=6$ and $H^3=2H^4_M=8$. So \ref{conj:bmt} holds for $(X, H)$ by Corollary \ref{cor:trivial-cor-basepoint-free} and the first claim in Theorem \ref{thm:cicy-in-fano}.

If $M=\PP^1\times \PP^3$, $X\in |-K_M|$ is a general divisor, and $H$ is the restriction of $(1,1)$-divisor on $M$, then we have $H^3=14$ and $\chi(\cO_X(H))=8$. As $H$ is very ample, \ref{conj:bmt} holds for $(X, H)$ by the second statement of Corollary \ref{cor:trivial-cor-basepoint-free}.
\end{Ex}

\subsection{Weighted CICY}\label{sec:hypergeo}

We now discuss one of the most important classes of smooth Calabi--Yau threefolds, which are quasi-smooth complete intersections in weighted projective spaces. There are $13$ deformation families, also known as hypergeometric one-parameter models in string theory.

\begin{Thm}\label{thm:hypergeo}
Let $X$ be a Calabi--Yau threefold that can be realized as a general quasi-smooth complete intersection in a weighted projective space and is not an intersection with a linear cone, and let $H$ be the ample generator of $\Pic(X)$. Then \emph{\ref{conj:bmt}} holds for $(X, H)$ and some $1$-cycle $\Gamma$ with $\Gamma.H\geq 0$.
\end{Thm}

\begin{proof}
According to \cite[Theorem 4.5]{chen:on-quasismooth-wci}, such $X$ falls into $13$ classes of complete intersections in weighted projective spaces. In the following, we denote by $X_{d_1,\dots,d_n}$ a general complete intersection of multidegree $(d_1,\cdots,d_n)$ in a fixed weighted projective space. If $X$ is smooth, then the restriction $H$ of the hyperplane class on the ambient weighted projective space is an ample divisor and generates $\Pic(X)$. We divide the proof into several cases as follows.

\begin{itemize}
   \item $X_5\subset \PP^4, X_{2,4}, X_{3,3}\subset \PP^5, X_{2,2,3}\subset \PP^6$, and $X_{2,2,2,2}\subset \PP^7$. For these examples, the result follows from Theorem \ref{thm:cicy-in-fano}.

    \item $X_8\subset \PP(1^4, 4)$, $X_6\subset \PP(1^4,2)$, $X_{3,4}\subset \PP(1^5,2)$, $X_{2,6}\subset \PP(1^5,3)$, and $X_{4,4}\subset \PP(1^4,2^2)$. For these threefolds, $H$ is basepoint-free. So the result can be deduced from Corollary \ref{cor:trivial-cor-basepoint-free}.

    \item $X_{10}\subset \PP(1^3,2,5)$. In this case, $2H$ is basepoint-free, and \ref{conj:bmt} holds for $(X_{10}, 2H)$ by Corollary \ref{cor:trivial-cor-basepoint-free}. Using Remark \ref{rmk:basepoint-free}, the result also holds for $(X_{10}, H)$.

    \item $X_{6,6}\subset \PP(1^2,2^2,3^2)$. In this case, $2H$ is basepoint-free. Moreover, a complete intersection curve $C$ of two general divisors in $|2H|$ is a smooth non-hyperelliptic curve of genus $9$. 
    By Lemma \ref{lem:mu/2-bound}, we have 
    \[
    \bn_C \leq \frac{9}{2}<5=\chi(\cO_X(2H)). 
    \]
    Therefore, Theorem \ref{thm:main-criterion}(a) implies \ref{conj:bmt} for $(X_{6,6}, 2H)$, and Remark \ref{rmk:basepoint-free} implies the result for $(X_{6, 6}, H)$.

    \item $X_{4,6}\subset \PP(1^3,2^2,3)$. We take general $S\in |\cO_X(2H)|$ and $C\in |\cO_S(2H_S)|$. Then $S=S_{4,6}\subset \PP(1^3,2,3)$ and $C=C_{4,6}\subset \PP(1^3,3)$ are both smooth since $2H$ is basepoint-free. Moreover, we have an embedding $C\hookrightarrow S_6\subset \PP(1^3,3),$ where $S_6$ is a smooth K3 surface and is a double cover of $\PP^2$ branched along a sextic curve. Further, since $X_{4,6}$ is general, $S_{6}$ has Picard rank one. Thus, \ref{conj:bmt} for $(X_{4,6}, 2H)$ follows from Proposition \ref{prop:K3} and Theorem \ref{thm:main-criterion} as $\bn_C\leq \sqrt{48}$ and $\chi(\cO_X(2H))=8$. By Remark \ref{rmk:basepoint-free}, this implies the result for $(X_{4, 6}, H)$.
\end{itemize}
Now the proof is completed.
\end{proof}

Combining Theorem \ref{thm:hypergeo} with Remark \ref{rmk-etale} also establishes \ref{conj:bmt} for certain non-simply connected examples. 

\begin{Ex}\label{ex:quotient}
Let $(X, H)$ be any polarised Calabi--Yau threefold that we considered in this section. Suppose there is a finite group $G$ that acts freely on $X$ and $H$ is $G$-equivariant. Then by Remark \ref{rmk-etale}, \ref{BG3} as well as \ref{conj:bmt} hold for $(X', H')$, where $X'\coloneqq X/G$ and $H'$ is induced by $H$. In particular, this applies to the example with non-abelian fundamental group containing a rigid ample surface, constructed in \cite{beauville:cy3-nonabelian} as a finite free quotient of $X_{2,2,2,2}\subset \PP^7$.
\end{Ex}

\subsection{Cyclic covers of Fano threefolds}

Given an integer $d\geq 2$ and a smooth Fano threefold $Y$ of index $r$ with a smooth divisor $$B\in \left|-\frac{d}{d-1}K_Y\right|,$$ we can construct a degree $d$ cyclic cover $X\to Y$ branched along $B$ such that $X$ is a simply connected Calabi--Yau threefold. Note that in this case, $d-1$ divides $r$ and $|B|$ is always basepoint-free by \cite[Corollary 2 and 2*]{ein-lazarsfeld:global-generation} and \cite[Corollary 3.2]{kawamata:fujita-3-4fold}. Under some assumptions on $Y$, we can prove \ref{conj:bmt} for such $X$.

\begin{Thm}\label{thm:double-cover}
Let $X$ be a Calabi--Yau threefold with a cyclic cover $\pi\colon X\to Y$ of degree $d\geq 2$ to a smooth Fano threefold $Y$ of index $r$. Set $H\coloneqq \pi^*H_Y$. Assume that the branch divisor $B$ is general when $H_Y$ is not basepoint-free. Then \emph{\ref{conj:bmt}} holds for $(X, H)$ and some $1$-cycle $\Gamma$ with $\Gamma.H\geq 0$ if 

\begin{itemize}
\item $r\geq 2$,

\item $r=1$ and $Y$ has Picard number one, or

\item $r=1$ and $H_Y^3\leq 6$ such that $Y$ is not of Picard number $10$.
\end{itemize}

\end{Thm}

\begin{proof}
To apply Theorem \ref{thm:main-criterion}(a), we first show that we can choose a smooth surface $S \in |H|$ and a smooth curve $C \in |H_S|$. 
By \cite[Theorem 2.3.1]{shafarevich:fano-varieties}, a general surface $S'\in |H_Y|$ is smooth. Since $|B|$ is basepoint-free, and since we assume that $B$ is general whenever $H_Y$ is not basepoint-free, we may assume that $D=B\cap S'$ is smooth. Consequently, $S\coloneqq \pi^{-1}(S')\in |H|$ is smooth. Now, we choose $C\in |H_S|$ as follows.

\begin{itemize}
    \item If $d-1=r$ and $(r,H_Y^3)\neq (1, 6)$, we take $C\coloneqq (\pi^{-1}(D))_{\mathrm{red}}$.
    
    \item If $d-1<r$ or $(r,H_Y^3)=(1, 6)$ and $Y$ is not of Picard number $10$, we choose a smooth curve $C' \in |H_{S'}|$ which intersects with $B$ transversely and define $C \coloneqq \pi^{-1}(C')$.
\end{itemize}

Such a curve $C'$ in the second case always exists. Indeed, when $d-1<r$ or $(r,H_Y^3)=(1, 6)$ and $Y$ is not of Picard number $10$, either $S'$ is Fano or $H_Y$ is basepoint-free by \cite[Theorem 2.4.5]{shafarevich:fano-varieties}. Hence, general divisors in $|H_{S'}|$ are smooth. If $|H_Y|$ is basepoint-free, then we may choose a smooth curve $C' \in |H_{S'}|$ that intersects $B$ transversely. 
If $|H_Y|$ is not basepoint-free, then we first fix a smooth curve $C' \in |H_{S'}|$ and then take a general $B$ so that it intersects the curve $C'$ transversely.

By \cite[Corollary 2.1.14]{shafarevich:fano-varieties} and \cite[Remark 4.1.7]{lazarsfeld:positivity-I}, we have:
\begin{itemize}
    \item If $r>1$ and $d-1 <r$, then $\chi(\cO_X(H)) = \chi(\cO_Y(H_Y)) = \frac{1}{2}H_Y^3r+2$.
    \item If $r>1$ and $d-1 =r$, then $\chi(\cO_X(H)) = \chi(\cO_Y(H_Y)) +1= \frac{1}{2}H_Y^3r+3$.
    \item If $r=1$, then $\chi(\cO_X(H)) = \chi(\cO_Y(H_Y)) +1= \frac{1}{2}H_Y^3+4$. 
\end{itemize}
On the other hand, Lemma \ref{lem:weak-bound} implies 
$$\bn_C \leq \frac{g(C)-1}{2} +1 = \frac{dH_Y^3}{2} +1.$$
This proves \eqref{eq:thm-bn-smooth} in the cases 
$(r, H^3_Y) = (4,1), (3,2), (1, \leq 4)$, as well as in the case $(r, d) = (2,2)$. 
Moreover, when $(r, H_Y^3)=(1,6)$ and $Y$ is not of Picard number $10$, we have $g(C)=13$ and $g(C')=4$, together with a double cover $C\to C'$. By Lemma \ref{lem:cast-severi}, we see $\gon(C)\geq 4$, which implies $\mathrm{Cliff}(C)\geq 2$. Now Lemma \ref{lem:mu/2-bound} gives $\bn_C\leq H_Y^3=6$ and \eqref{eq:thm-bn-smooth} follows.

Next, we treat the remaining cases, which all satisfy $d-1=r$. Recall that by our choice of $C$, we have an isomorphism $C \cong B \cap S'$ under the morphism $\pi$. 

When $(r,d)=(2,3)$, by the construction above, $C$ can be embedded into a smooth del pezzo surface $S'$ as a divisor in $|3H_{S'}|=|-3K_{S'}|$. Applying Proposition \ref{prop:delpezzo} to $C\hookrightarrow S'$, we get $\bn_C\leq \max\left\{1+H_Y^3, 3\right\},$ which implies \eqref{eq:thm-bn-smooth}.

When $(r, H_Y^3)=(1, > 6)$ and $Y$ has Picard number one, $C\in |H_S|$ can be embedded into a smooth K3 surface $S'$ as a divisor in $|2H_{S'}|$. It is known that $H_Y$ is very ample for such $Y$ by \cite[Proposition 4.1.11]{shafarevich:fano-varieties}. Therefore, by Noether--Lefschetz theorem and taking $S'$ to be very general, we have $\Pic(S')=\Z H_{S'}$. Hence, we can apply Proposition \ref{prop:K3} to $C\hookrightarrow S'$ and obtain \eqref{eq:thm-bn-smooth}.
\end{proof}

Although we expect that a version of the stronger BG inequality as \ref{BG3} holds for any simply connected Calabi--Yau threefolds with a suitable polarisation, we do not expect \eqref{eq:thm-bn-smooth} to be true in full generality.

\begin{Ex} \label{ex:pathology}
When a Fano threefold $Y$ has Picard number $10$, it is of the form $S_1\times \PP^1$, where $S_1$ is a del Pezzo surface of degree $1$. Let $\pi\colon X\to Y$ be a double cover branched along a general divisor $B\in |-2K_Y|$. In this case, we claim that the choice of $H_Y=-K_Y$ and $H=\pi^*H_Y$ does not satisfy the inequality \eqref{eq:thm-bn-smooth} in Theorem \ref{thm:main-criterion}. Indeed, we have $$\chi(\cO_X(H))=\chi(\cO_Y(H_Y))+1=7$$ as above. Since $-2K_Y$ is basepoint-free and $B$ is general, the same argument as the case $(r, H_Y^3)=(1,4)$ in Theorem \ref{thm:cicy-in-fano} shows that we can choose a smooth surface $S\in |H|$ and a smooth curve $C\in |H_S|$. However, the natural map $f\colon C\to S_1\times \PP^1\to \PP^1$ is of degree $2$ by the computation of intersection numbers. Therefore, Lemma \ref{lemma:bnHyperelliptic} implies $\bn_C=7=\chi(\cO_X(H))$.
\end{Ex}

It is clear that the failure of \eqref{eq:thm-bn-smooth} in the above example comes from the existence of a morphism $C\to \PP^1$ of degree two (cf.~Lemma \ref{lemma:bnHyperelliptic}). Therefore, we expect that the pathological phenomenon described above is exceptional and will not happen when $H$ is very ample.

\begin{Con}\label{conj:bnC}
Let $(X, H)$ be a polarised Calabi--Yau threefold with $H$ very ample and $\Pic(X)=\Z H$. Then there exists a smooth surface $S \in |H|$ and a smooth curve $C \in |H_S|$ such that
\[\bn_C < \chi(\cO_X(H)).\]
\end{Con}
The geometry of $X$ and $S$ may come into the study of Conjecture \ref{conj:bnC}. For example, the Jacobian $J(C)$ of a generic $C\in |H_S|$ satisfies $\mathrm{End}(J(C))\cong \mathbb Z$ by \cite[Theorem~(1.1)]{hyperplane} and $\mathrm{Aut}(C)=\{\mathrm{id}_C\}$, which may lead to a strong restriction on $\bn_C$.

\bibliography{mybib}                     
\bibliographystyle{halpha}

\end{document}